\documentclass{oblivoir}
\usepackage{amssymb, amsmath, amsthm}
\usepackage[inline]{enumitem}
\usepackage{titlesec}
\usepackage{geometry}
\usepackage{indentfirst}
\usepackage{xparse}
\usepackage{thmtools}
\usepackage[nameinlink]{cleveref}
\usepackage{tikz-cd}
\usepackage{stmaryrd}
\usepackage[style=alphabetic, backend=bibtex]{biblatex}
\usepackage{csquotes}
\usepackage{url}
\usepackage{hyperref}
\usepackage{graphicx}

\setlist{  
  wide,
  itemindent=0.2cm
}
\everymath{\displaystyle}

\newcommand{\ZZ}{\mathbb{Z}}
\newcommand{\NN}{\mathbb{N}}
\newcommand{\QQ}{\mathbb{Q}}
\newcommand{\Qbar}{\overline{\mathbb{Q}}}
\newcommand{\CC}{\mathbb{C}}
\newcommand{\AAA}{\mathbb{A}}
\newcommand{\PP}{\mathbb{P}}

\newcommand{\JJ}{\mathbb{J}}
\newcommand{\kk}{\mathbf{k}}
\newcommand{\rr}{\mathbf{r}}
\newcommand{\xx}{\mathbf{x}}

\newcommand{\calK}{\mathcal{K}}
\newcommand{\GL}{\mathrm{GL}}
\newcommand{\SL}{\mathrm{SL}}
\newcommand{\dom}{\mathrm{dom}}

\newcommand{\frakI}{\mathfrak{I}}

\NewDocumentCommand{\bideg}{m}{\overline{\mathrm{deg}}\,#1}

\declaretheorem[style=definition, numberwithin=section, name=]{justnumber}
\declaretheorem[qed=$\square$, sibling=justnumber, name=Theorem]{theo}
\declaretheorem[sibling=justnumber, name=Theorem]{theop}

\declaretheorem[sibling=justnumber, name=Lemma]{lemmap}
\declaretheorem[qed=$\square$, sibling=justnumber, name=Proposition]{prop}
\declaretheorem[sibling=justnumber, name=Proposition]{propp}
\declaretheorem[qed=$\square$, sibling=justnumber, name=Corollary]{coro}
\declaretheorem[sibling=justnumber, name=Corollary]{corop}

\declaretheorem[style=definition,qed=$\square$, sibling=justnumber]{definition}
\declaretheorem[style=definition,qed=$\square$, sibling=justnumber, name=Notation]{notation}
\declaretheorem[style=definition, qed=$\square$, sibling=justnumber, name=Remark]{rmk}

\declaretheorem[qed=$\square$, name=Theorem]{theom}

\DeclareMathOperator{\Spec}{Spec}
\DeclareMathOperator{\Supp}{Supp}

\DeclareMathOperator{\Aut}{Aut}
\DeclareMathOperator{\Aff}{Aff}

\DeclareMathOperator{\Bir}{Bir}
\DeclareMathOperator{\End}{End}

\DeclareMathOperator{\id}{id}

\title{Algorithm to Compute Orbit Zariski Closure in Affine Plane}
\author{Young Joon, Ley}
\date{}

\begin{document}

\renewcommand{\abstractname}{Abstract}

\maketitle

\begin{abstract}
The article demonstrates the procedure how to compute the Zariski closure of an orbit by an algebraic action of finitely generated group on the affine plane. First half of the algorithm is about deciding whether given finitely generated group is contained in an algebraic group. For the next half, we compute the totality of the invariant subvarieties for a single triangular automorphism. Then, the computation for the individual generators is applied to compute the orbit Zariski closure for a finitely generated group. 
\end{abstract}

\tableofcontents

\section{Introduction}
\label{Section:Introduction}

Let us relate the problem of computing the orbit Zariski closure with the classical orbit problem. When a group $G$ acts on a set $X$, the orbit problem is the problem to decide whether, given $x,y\in X $, there exists $g\in G $ such that $g.x=y $, and to compute such $g$ if it exists.

As a positive answer towards the orbit problem, there is the Grunewald-Segal Algorithm \cite[Algorithm A]{Grunewald-Segal}, which proves the decidability of orbit problem for arithmetic group when the action of the arithmetic group is induced from its ambient algebraic group. More precisely, suppose $X\subset \AAA^n_{\QQ}$ is an affine variety defined over $\QQ $, $G$ is an algebraic group defined over $\QQ $, and $\rho:G\rightarrow \Aut_{\QQ}(X) $ is a representation of $G$ in the algebraic automorphism group of $X$ over $\QQ $. The representation $\rho $ need not be an algebraic morphism, and in the cited paper, $X, G, \rho $ are all assumed to be ``explicitly given''. Suppose $\Gamma \subset G $ is an arithmetic subgroup which is also ``explicitly given''. Then, the Grunewald-Segal Algorithm proves that the orbit problem is decidable for the induced representation $\rho:\Gamma\rightarrow \Aut_{\QQ}(X) $. 

In comparison with the Grunewald-Segal Algorithm, if we drop the arithmeticity assumption from $\Gamma $, then we cannot guarantee the decidability of the orbit problem. The counterexample is provided by Mikhailova \cite{Mikhailova}, where the author proves that there exists a finitely generated subgroup $\Gamma \subset \SL_4(\ZZ) $ for which the the conjugacy problem is undecidable. Thus, there exists a representation $\rho:\Gamma \rightarrow \Aut_{\QQ}(\AAA^{16}) $ for which the orbit problem is undecidable. 

Restricting the general orbit problem in group theory to the algebraic dynamics, we can ask the following question:
\begin{quote}
\label{quote:alg-orbit-problem-hard}
Suppose a finitely generated group $\Gamma $ acts on an algebraic variety $X$, i.e. there exists a representation $\Gamma \rightarrow \Aut_{\mathrm{Sch}}(X) $. For given geometric points $x,y\in X $, decide whether there exists $g\in \Gamma$ such that $g.x=y $. 
\end{quote}
The Grunewald-Segal Algorithm \cite{Grunewald-Segal}, and the counterexample of Mikhailova \cite{Mikhailova} respectively provide the positive and negative answers to special instances of the problem.

Above problem is just as unapproachable as the general orbit problem in group theory is. It is tempting to make the problem more approachable by relieving the orbit condition. Thus, an ``orbit closure problem'' might be suggested:
\begin{quote}
\textbf{Orbit Closure Problem:} Suppose a finitely generated group $\Gamma$ acts algebraically on an algebraic variety $X$, i.e. there exists a representation $\Gamma \rightarrow \Aut_{\mathrm{Sch}}(X) $. For given geometric points $x,y \in X$, decide whether $y\in \overline{\Gamma.x}$. 
\end{quote}
Thus, the computation of an orbit Zariksi closure is to be thought as an ``algebraically approximate'' answer towards the orbit problem.

In the introduction of the paper \cite{Whang}, Whang asks whether it is always possible to compute the orbit Zariski closure of a point under a set of of endormophisms of the algebraic variety. In the present article, we propose a positive answer towards the simplest nontrivial case when a finitely generated group acts algebraically on $\AAA^2_{\Qbar} $.

The main result of the paper is the following theorem (\Cref{theo:orbit-zariski-closure}):

\begin{theom}
\label{main-theo:ozc-algo}
Suppose $\Gamma \rightarrow \Aut(\AAA^2_{\Qbar}) $ is a representation of a finitely generated group $\Gamma $. Let $x\in \AAA^2_{\Qbar} $ be a geometric point. Then, there exist an algorithm to compute the Zariski closure of the orbit $\Gamma.x $. Thus, the orbit closure problem is decidable in $\AAA^2 $. 
\end{theom}

\par We gather relevant results in arithmetic and algebraic dynamics together to elaborate the algorithm. In \Cref{section:blanc-stampfli}, we compute the orbit Zariski closure when $\Gamma \subset \Aut(\AAA^2_{\Qbar}) $ is unbounded in the degree. The key ingredients are the Blanc-Stampfli's Theorem \cite[Theorem 1]{blanc-stampfli}, and the uniform orbit bound theorem of Whang \cite[Theorem 1.2]{Whang}. The Blanc-Stampfli's Theorem implies that a finitely generated subgroup $\Gamma \subset \Aut(\AAA^2) $ with unbounded degree cannot preserve a curve. The Whang's Uniform Orbit Bound Theorem gives effective bound for the size of a finite orbit by a finitely generated group action on an algebraic variety. 

Thus, combining these two theorems, we easily obtain the algorithm to compute the orbit Zariski closure for unbounded $\Gamma \subset \Aut(\AAA^2) $. It remains to decide whether $\Gamma $ is bounded degree or not, then compute the orbit Zariski closure for the cases when $\Gamma $ is bounded in degree.

\par The celebrated Jung-van der Kulk Theorem states that the group $\Aut(\AAA^2) $ is the amalgamated product of the following two subgroups \cite{Jung1942}, \cite{vdK1953}:
\begin{gather*}
  \Aff= \{(a_{11}x+a_{12}y+a_{13}, a_{21}x+ a_{22}y+a_{23} )\in \Aut(\AAA^2): a_{11}a_{22}-a_{12}a_{21}\neq 0 \}, \\
  \JJ= \{(ax+P(y), by+c)\in \Aut(\AAA^2): a\neq 0, \, b\neq 0,\, c\in \kk, \,P(y)\in \kk[y]  \} \, . 
\end{gather*}
Word length of an element in an amalgamated product group is well-defined. It turns out that group theoretic notion of length translates well to the notion of degree in the polynomial group $\Aut(\AAA^2)$. 
Serre's Theorem \cite[I.4.3. Theorem 8]{serre-trees} on the subgroups of bounded length in an amalgamated product group translates to the $\Aut(\AAA^2) $ via the Jung-van der Kulk Theorem. In the setting of the polynomial group, the Serre's Theorem tells that a finitely generated subgroup of $\Aut(\AAA^2) $ is of bounded degree if and only if it is conjugate into $\Aff $ or $\JJ $.

\par In \Cref{section:grp-theory}, we state the criterion to decide whether a finitely generated group in an amalgamated is of bounded length or not. The proof of the criterion is elementary and combinatoric. The criterion (\Cref{main-criterion}) is as follows:
\begin{theom}
\label{main-theo:criterion-amalgam}
  Suppose $G=A*_{A\cap B} B$ is an amalgamated product and $H=\langle g_1,\cdots, g_n\rangle \leq G $ a finitely generated subgroup. Suppose $g_1 $ is the generator with maximal length, i.e. $l(g_1)=\max_{j=1,\cdots, n} \{l(g_j) \} $. Then:
\begin{enumerate}[label=(\roman*)]
  \item $H$ is conjugate to a subgroup of one of the factors if and only if the generators $g_1,\cdots, g_n$ and the products $g_{1} g_j$ for all $j\neq 1 $ are bounded elements.
  \item Suppose that $H$ is conjugate into one of the factors and that $l(g_{1})>1 $. Then, $g_{1} $ has odd length. Let $g_{1}=k_1\cdots k_{2s+1} $ be a reduced expression of $g_{1} $ where $s \geq 1 $. Then, $(k_1\cdots k_s)^{-1} H (k_1\cdots k_s) $ is contained in one of the factors.
\end{enumerate}
\end{theom}

\par \Cref{section:conjugacy-aut-of-affine} is devoted to translating the results on general amalgamated product to the setting of polynomial group $\Aut(\AAA^2) $. Thus, exploiting the amalgamated product structure of $\Aut(\AAA^2) $, we obtain the following subalgorithm of the main algorithm, which is the content of \Cref{theo:conj-criterion-aut}:
\begin{theom}
\label{main-theo:main-criterion}
Let $H\leq \Aut(\AAA^2) $ be a finitely generated subgroup. Then, there exists the algorithm to decide whether $H $ is conjugate into $\Aff $ or $\JJ $ and compute the conjugator if so.
\end{theom}

\par Once we have established the conjugacy criterion for finitely generated subgroups in $\Aut(\AAA^2) $, we can focus on the computation of the orbit Zariski closure for the finitely generated, bounded subgroups in $\Aut(\AAA^2) $. \Cref{section:invar-curves} is devoted to this computation. 

In \Cref{section:invar-curves}, we compute the totality of the invariant subvarieties for given $\phi\in \JJ $. The computation of possible invariant irreducible curves in $\AAA^2$ already appeared in \cite[Section 4]{blanc-stampfli}. In the present article, we need the totality of the invariant subvarieties whch may not be irreducible. The information of the totality of the invariant subvarieties is contained in the invariant subvariety lattice $\frakI_{\phi} $ (\Cref{def:invar-subvar-lattice}), then encoded by the associated equivariant map of $\phi $ (\Cref{def:invar-subvar-lattice-classification}). 

The approach we take to develop the algorithm is that if we know the totality of the invariant subvarieties for each $\phi_1,\cdots, \phi_n \in \Aut(X) $, then we can compute the invariant subvarieties of the finitely generated subgroup $\langle \phi_1,\cdots, \phi_n \rangle $ by theorems relating the invariant subvariety lattices of each $\phi_i $. The lemma of the spirit is the \Cref{lemma:equivar-affine} and \Cref{lemma:equivar-projective}, which are very simple in the case of $\AAA^2 $. The lemmas serve to complete the main algorithm (\Cref{theo:orbit-zariski-closure}) in \Cref{section:computation-orbit-closure}.

\par The algorithm can be improved much efficiently in the case when the gerenators are algebraic elements, i.e. an element which in an image of homomorphism $G\rightarrow \Aut(\AAA^2) $ where $G$ is an algebraic group. The procedure reduces 2-dimensional dynamics to 1-dimensional dynamics and the 1-dimensional dynamics can be carried out much more efficiently. The optimization is not considered in the present article, and we focus on the existence of the algorithm.

\subsection*{Notations/Conventions}
\label{Intro:Conventions}
Throughout the article, $\kk $ denotes an arbitrary field. An affine space $\AAA^n $ is over the field $\kk$ unless otherwise mentioned.

\par By a \textit{variety over $\kk $}, we mean a reduced, separated, scheme finite type over $\kk$. It doesn't have to be irreducible, connected or pure dimensional. \textit{Curve} means 1-dimensional variety.

\par An element $\phi\in \Aut(\AAA^n) $ is denoted $\phi=(f_1,\cdots, f_n) $ where $f_1, \cdots, f_n \in \kk[x_1,\cdots, x_n] $ to mean that $f_1,\cdots, f_n$ are the components of $\phi $. There are two types of degrees that is used in this article. One is the usual degree $\deg{\phi}:= \max_{1\leq i \leq n}\{\deg{f_i}\} $ where $\deg{f}=\max_{a_\rr\neq 0} \{\sum_{1 \leq i \leq n} r_i \} $ if $f=\sum_{\rr} a_{\rr} \xx^{\rr}=\sum_{\rr} a_{\rr} x_1^{r_1}\cdots x_n^{r_n} $. The other is the \textit{bidegree} $\bideg{\phi}:=(\deg{f}, \deg{g}) $. 

\par A subgroup $H \subset \Aut(\AAA^2) $ is said to \textit{be of bounded degree}, if the set $\{\deg{\phi} \,\vert\, \phi\in H \} $ is bounded. It turns out that the boundedness of the degree is equivalent to the boundedness in the amalgamated product structure of $\Aut(\AAA^2) $ \Cref{prop:bdd-len-criterion}. Hence, we will often simply say that the subgroup is bounded or unbounded.

\par The subgroups $\Aff$ and $\JJ \subset \Aut(\AAA^2)$ denotes the affine linear group and de Jonqui\`{e}res group respectively:
\begin{gather*}
\Aff= \{(a_{11}x+a_{12}y+a_{13}, a_{21}x+ a_{22}y+a_{23} )\in \Aut(\AAA^2): a_{11}a_{22}-a_{12}a_{21}\neq 0 \}, \\
\JJ= \{(ax+P(y), by+c)\in \Aut(\AAA^2): a\neq 0, \, b\neq 0,\, c\in \kk, \,P(y)\in \kk[y]  \} \, . 
\end{gather*}
The subgroup $\JJ_n \leq \JJ $ is the subgroup consisting of automorphisms of degree $\leq n $, i.e. 
\[
  \JJ_n := \{(ax+ P(y), by+c)\in \JJ : \deg{P(y)}\leq n \} \, . 
\]
In particular, it is an algebraic group. 

\par For an $R$-scheme $X$, the scheme automorphism group of $X$ is denoted $\Aut_R(X)$, and the subscript is omitted if the base scheme is clear. For a curve $C$, $\Aut(\AAA^2, C) \subset \Aut(\AAA^2)$ denotes the subgroup preserving curve $C$. If $X$ is a surface and $B\subset X $ a curve, then $\Bir(X,B) $ denotes the group of birational self-maps of $X$ preserving $B$. If $C\subset X $ is another curve, then $\Bir((X,B), C) $ denotes the subgroup of $\Bir(X,B) $ consisting of the birational maps that preserve $C$.

\section{When $\Gamma \subset \Aut(\AAA^2)$ is Unbounded}
\label{section:blanc-stampfli}

\par In this section, we describe how to compute the orbit Zariski closure given that we know the finitely generated group $\Gamma \subset \Aut(\AAA^2) $ is unbounded in degree. The algorithm is a simple consequence of the Blanc-Stampfli's Theorem (\Cref{coro:blanc-stampfli}) and the Whang's Uniform Orbit Bound Theorem (\Cref{prop:whang}).

Following is the Blanc-Stampfli's Theorem:

\begin{prop}[\protect{\cite[Theorem 1]{blanc-stampfli}}]
\label{prop:blanc-stampfli}
Suppose $C\subset \AAA^2 $ is a curve. Then $\Aut(\AAA^2, C) $ is conjugate to a subgroup of $\JJ $ or $\Aff $. 
\end{prop}

\begin{coro}
\label{coro:blanc-stampfli}
Suppose a finitely generated group $\Gamma$ acts on $\AAA^2 $ and preserves a curve. Then $\Gamma$ is of bounded degree in $\Aut(\AAA^2) $.
\end{coro}

\par The proof uses the techniques of the affine algebraic geometry; we embed the affine plane $\AAA^2 $ into $\PP^2 $ and complete the curve $C\subset \AAA^2$ in $\PP^2 $. The completed curve will be denoted by the same $C$. Any automorphism in $\Aut(\AAA^2) $ extends to a birational map $\Bir(\PP^2, B_{\PP^2}) $ where $B_{\PP^2} $ denotes the line at infinity $\{z=0\} $, i.e. $\Aut(\AAA^2)= \Bir(\PP^2, B_{\PP^2}) $. The birational self-maps extending the morphisms in $\Aut(\AAA^2, C) $ should preserve all singularities of $C$ lying on the line at infinity. Then, it is possible to precisely determine the group $\Aut(\AAA^2, C)= \Bir((\PP^2,B_{\PP^2}), C)$ by using the theory of links between the rational ruled surfaces and $\PP^2 $. The theory of links is an instance of Sarksisov Theory for rational surfaces, which factorizes the birational self-maps of $\PP^2 $ as composition of simple birational maps between the ruled surfaces and $\PP^2 $.

\par Recall that an \textit{ind-variety} $\mathcal{V} $ is defined as a sequence of closed immersions $\mathcal{V}_1 \subset \mathcal{V}_2 \subset \mathcal{V}_3 \subset \cdots $ of algebraic varieties, and ind-group is an ind-variety with a compatible group structure \cite[Definition 1.1.1]{Furter-Kraft}. In view of the invariant theory of ind-groups, \Cref{prop:blanc-stampfli} can be thought as a statement that invariant theory of ind-groups is trivial on $\AAA^2 $.

\begin{corop}
Let $\Gamma \subset \Aut(\AAA^2) $ be a finitely generated group. If Zariski closure $\overline{\Gamma} $ is an ind-group which is not an algebraic group, then the invariant ring by $\Gamma$ is $\kk[x,y]^{\Gamma}= \kk $. 
\end{corop}
\begin{proof}
That $\overline{\Gamma} $ is not an algebraic group is equivalent to saying that $\Gamma $ is not of bounded degree. If there exists some nontrivial invariant $f\in \kk[x,y] $, then the invariant polynomial gives an invariant curve $f=0$ in $\AAA^2 $ and it contradicts \Cref{coro:blanc-stampfli}. 
\end{proof}

\begin{prop}[\protect{\cite[Theorem 1.4]{Whang}}]
\label{prop:whang}
Let $S$ be a finite set of endomorphisms of an algebraic variety $V/\Qbar $. There is an algorithm to decide, given $x\in V(\Qbar) $, whether or not $x$ is $S$-periodic. 
\end{prop}

\par Summing up the Blanc-Stampfli's Theorem (\Cref{coro:blanc-stampfli}) and Whang's Uniform Bound Theorem (\Cref{prop:whang}), we compute the orbit Zariski closure for unbounded $\Gamma $ by:

\begin{corop}
\label{coro:unbdd-ozc}
Suppose $\Gamma \subset \Aut(\AAA^2_{\Qbar}) $ is unbounded in degree. Then, given a point $p\in \AAA^2(\Qbar) $, there exists an algorithm to compute the orbit Zariski closure of $p$. 
\end{corop}
\begin{proof}
Since $\Gamma $ preserves no curve in $\AAA^2 $, the orbit Zariski closure $\overline{\Gamma.p}$ cannot be a curve, or else it contradicts \Cref{coro:blanc-stampfli}. Thus, the dimension of $\overline{\Gamma.p} $ is either 0 or 2. 

By using the algorithm of \Cref{prop:whang}, we can decide whether or not the dimension of $\overline{\Gamma.p} $ is 0; if it is periodic if and only if it is 0-dimensional. If it is 0-dimensional, then $\overline{\Gamma.p}=\Gamma.p $ has cardinality less than the bound of \Cref{prop:whang}, and we compute the set $\Gamma.p $ by simply enumerating all its points. 

If $\overline{\Gamma.p} $ is not 0-dimensional, then $\overline{\Gamma.p}=\AAA^2 $.  
\end{proof}

\section{Conjugacy Problem in Amalgamated Free Product}
\label{section:grp-theory}

\par In this section, I collect some basic facts on amalgamated product and prove a criterion to decide whether a finitely generated subgroup in the amalgamated product is conjugate to a subgroup of one of the factors. 

\par For a group $G$ and its subgroups $A$ and $B$, suppose $G$ has the structure of the amalgamated free product $G=A*_{A\cap B}B $. More precisely, it means that the subgroups $A$ and $B$ of $G$ satisfy the following isomorphism of groups:
\[
  G\cong A*_{A\cap B}B:=\langle a\in A,b\in B \, | \, ab^{-1}=1 \text{ if $a= b$ in $A\cap B $}  \rangle  \, . 
\]

\par The subgroups $A$ and $B$ are called the \textit{factors} of the amalgamated product $G$. In this section, $G$ is always the amalgamated product over the intersection $A*_{A \cap B} B $.

\begin{definition}
  For $g\in G$, suppose $g=g_1\cdots g_r$ is such that $g_i$ and $g_{i+1}$ lie in different factors for all $i$, and none of $g_i$ is contained in the intersection $A\cap B$ of the factors. Then, the expression $g=g_1\cdots g_r$ is called the \textit{reduced expression}.
\end{definition}

\par The following proposition states that the reduced expression enjoys a uniqueness property modulo $A\cap B $. I refer the proof to any book on combinatorial group theory, for example \cite{Magnus} or \cite{Schupp}.

\begin{prop}
\label{reduced-expr-unique}
\protect{\cite[Corollary 4.4.2]{Magnus}}
Let $g\in G- (A\cap B) $. If $g=g_1 \cdots g_r = g_1' \cdots g_s' $ are two distinct reduced expressions of $g$, then $r=s $. Moreover, there exist $c_1=1, c_2,\cdots, c_r\in A\cap B $ such that $g_r'=c_r g_r $ and $g_i' c_{i+1}=c_i g_i $ for $i=1,\cdots, r-1 $. This implies that $g_i $ and $g_i' $ are in the same factor for all $i=1,\cdots, r $.
\end{prop}

In virtue of \Cref{reduced-expr-unique}, we define the following. 

\begin{definition}
Let $g\in G- (A \cap B)$, and suppose $g= g_1 \cdots g_r $ is a reduced expression of $g$. We define $r$ as the \textit{length of $g$}, denoted by $l(g)=r$. This definition is well-defined by \Cref{reduced-expr-unique}, independent of the chosen reduced expression of $g$. If $g\in A\cap B$, then we set $l(g)=0$.

\par If a subgroup $H\leq G$ is such that the length of elements in $H$ is bounded, then $H$ is said to be \textit{bounded}. If $g\in G$ is such that $\langle g \rangle \leq G$ is bounded, then $g$ is said to be \textit{bounded}.
\end{definition}

\par If we have chosen a fixed set of coset representatives for $A\cap B $ in $A$ and $B$ respectively, the uniqueness of \Cref{reduced-expr-unique} can be stated as follows:

\begin{prop}
\label{prop:normal-form}
\cite[Theorem 4.4]{Magnus} 
Let $\calK_A $ and $\calK_B $ denote sets of right coset representatives for $(A\cap B)\backslash A $ and $(A\cap B) \backslash B $, respectively. Furthermore, suppose the identity $1$ belongs to both $\calK_A $ and $\calK_B $. Then, any element $g\in G $ possesses a unique reduced expression $g= hg_1 \cdots g_r $ satisfying the following conditions:
\begin{enumerate}[label=(\roman*)]
\item $h\in A\cap B $,
\item each $g_i $ belongs to $\calK_A - \{1 \} $ or $\calK_B - \{ 1\} $ for each $i=1,\cdots, r$,
\item $g_i$ and $g_{i+1} $ reside in different factors for each $i=1,\cdots, r-1$.
\end{enumerate}
\end{prop}

The decomposition in the \Cref{prop:normal-form} is called the \textit{normal form} with respect to the chosen representatives $\calK_A $ and $\calK_B $.

\par Now, we state the criterion to decide whether a given finitely generated group $H\leq G$ is conjugate into one of the factors. The theorem also suggests the method to actually find the conjugator when factorization problem is decidable in $G$. I temporarily postpone the proof until I have introduced the notion of \textit{cyclically reducedness}. 

\begin{theo}
\label{main-criterion}
Suppose $G=A*_{A\cap B} B$ is an amalgamated product and $H=\langle g_1,\cdots, g_n\rangle \leq G $ a finitely generated subgroup. Suppose $g_1 $ is the generator with maximal length, i.e. $l(g_1)=\max_{j=1,\cdots, n} \{l(g_j) \} $. Then:
\begin{enumerate}[label=(\roman*)]
\item $H$ is conjugate to a subgroup of one of the factors if and only if the generators $g_1,\cdots, g_n$ and the products $g_{1} g_j$ for all $j\neq 1 $ are bounded elements.
\item Suppose that $H$ is conjugate into one of the factors and that $l(g_{1})>1 $. Then, $g_{1} $ has odd length. Let $g_{1}=k_1\cdots k_{2s+1} $ be a reduced expression of $g_{1} $ where $s \geq 1 $. Then, $(k_1\cdots k_s)^{-1} H (k_1\cdots k_s) $ is contained in one of the factors.
\end{enumerate}
\end{theo}

\par The theorem is version of Serre's theorem for finitely generated subgroup.

\begin{prop}[\protect{\cite[I.4.3. Theorem 8.]{serre-trees}}]
\label{serre-tree}
Subgroup of $G$ is conjugate to a subgroup of $A$ or $B$ if and only if it is bounded. In particular, an element $g\in G $ is conjugate an element in $A$ or $B$ if and only it is a bounded element.
\end{prop}

\par Following notion is useful in dealing with conjugacy problem in the amalgamated product.

\begin{definition}
\cite[Section 4.2]{Magnus}
An element $g\in G$ is called \textit{cyclically reduced} if $l(g)\leq 1 $ or $l(g) $ is even. These are the elements which start and end with letters from distinct factors. For $g\in G $, if $\overline{g}\in G $ is conjugate with $g$ and is cyclically reduced, then $\overline{g} $ is called the \textit{cyclically reduced form} of $g$.
\end{definition}

The proof of the proposition below tells how to deduce the cyclically reduced form of given $g \in G $.

\begin{propp}[\cite{Magnus}]
\label{sec4-exist-cyclic-reduced}
Every $g\in G$ has a cyclically reduced form.   
\end{propp}
\begin{proof}
  Suppose $g=g_1\cdots g_{r} \in G$ is in its reduced expression but is not cyclically reduced. Therefore, $r\geq 3 $, and $g_1$ and $g_r$ reside in the same factor, say $g_1, g_2 \in A- B $. One may attempt to obtain the cyclically reduced form of $g$ by conjugating it with $g_1^{-1}$, i.e., $g_1^{-1}g g_1= g_2\cdots g_{r-1} g_{r} g_1 $. If $g_r g_1 \notin A\cap B $, then $g_2 \in B- A$ while $g_rg_1 \in A-B $, and $g_1^{-1}g g_1=g_2\cdots g_{r-1}(g_r g_1) $ is a reduced expression of length $r-1 $. Hence, we arrive at the cyclically reduced form of $g$. 
  
  \par However, if $g_r g_1 \in A\cap B $, then the reduced expression is $g_1^{-1} g g_1= g_2\cdots g_{r-2} (g_{r-1} g_r g_1)$, which has length $r-2 $. It is not cyclically reduced unless $r-2=1$. In that case, we can again conjugate $g_1^{-1}g g_1 $ by the first letter and repeat the procedure until we reach the cyclically reduced form. 
\end{proof}

Although the cyclically reduced form of $g$ is not unique, the proposition below states that the length of the cyclically reduced form is uniquely determined. Moreover, if the length of the cyclically reduced form is greater than $1$, then the cyclically reduced form is unique up to cyclic permutation and conjugation by an element of $A\cap B$.

\begin{propp}[\cite{Magnus}]
\label{sec4-unique-cyc-reduced}
Suppose $g\in G $. Let $\overline{g}, \overline{g}'$ be the cyclically reduced forms of $g$. 
\begin{enumerate}[label=(\roman*)]
  \item If $l(\overline{g})\leq 1 $, then $l(\overline{g}')\leq 1 $.
  \item If $l(\overline{g})\geq 2 $, then $l(\overline{g})=l(\overline{g}') $. Moreover, $\overline{g}' $ is obtained by cyclic permutation of the letters of $\overline{g} $ and then conjugation by an element of $A\cap B $. More precisely, if $\overline{g}=g_1\cdots g_r $ is a reduced expression of $\overline{g} $, then there exists a cyclic permutation $\sigma\in S_r $ and some $h\in A\cap B $ such that $\overline{g}'=hg_{\sigma(1)}\cdots g_{\sigma(r)} h^{-1} $. 
\end{enumerate}
\end{propp}
\begin{proof}
I will just prove (ii), and the proof of (i) is done by the same argument with a little modification of the indices. Since $\overline{g} $ and $\overline{g}' $ are both cyclically reduced forms of $g$, there exists an element $a\in G$ such that $a\overline{g}a^{-1}= \overline{g}' $. For simplicity, assume $l(a)=1 $, and the general case easily follows by induction on the length of the conjugator. 

\par Since $a\overline{g} a^{-1}= ag_1\cdots g_r a^{-1} $ is cyclically reduced, either $g_1 $ or $g_r $ is in the same factor with $a$, and $a$ cancels out with either the first or the last letter of $\overline{g}$. Suppose $g_r$ and $a$ are in the same factor and that $l(g_r a^{-1})=0$. Denote by $g_r a^{-1}=h $. Then, 
\[
  a\overline{g} a^{-1} = h^{-1} g_r g_1\cdots g_{r-1} h
\]
is a reduced expression, which is indeed cyclic permutation of letters of $\overline{g} $ and conjugation by an element $h\in A\cap B$.  
\end{proof}

\begin{rmk}
\label{rmk:cyc-reduced}
\begin{enumerate}[label=(\roman*)]
\item Suppose $g\in G $ is cyclically reduced and $l(g)>1 $. Then, it is unbounded since $l(g^n)=n l(g) $, and cannot be conjugated into one of the factors by \Cref{serre-tree}.

\item If $g\in G $ can be conjugated into one of the factors, then there exists a reduced expression of the form $g= h_1\cdots h_r \overline{g} h_r^{-1}\cdots h_1^{-1} $. Indeed, suppose $g=h \tilde{g} h^{-1} $ be such that $l(\tilde{g})\leq 1 $, and let $h= h_1\cdots h_n $ be a reduced expression of $h$. If $h_1\cdots h_n \tilde{g} h_n^{-1}\cdots h_1^{-1} $ is not a reduced expression, then the only possibility is that $h_n$ and $\tilde{g}$ are in the same factor, and $h_n \tilde{g} h_n^{-1} $ is of length $\leq 1$. If 
\[
  h_1\cdots h_{n-1}(h_n \tilde{g} h_n^{-1}) h_{n-1}^{-1}\cdots h_1^{-1}  
\]
is still not a reduced expression, then it forces $h_n \tilde{g} h_n^{-1}\in A\cap B $, and $l(h_{n-1} h_n \tilde{g} h_n^{-1} h_{n-1}^{-1})\leq 1 $. After merging sufficiently many letters around $\tilde{g} $, we arrive at the reduced expression of the form $g=h_1 \cdots h_r \overline{g} h_r^{-1} \cdots h_1^{-1} $. 

\par Likewise, if the cyclically reduced form of $g' \in G $ is $g_1'\cdots g_k' $ which has length $k>1 $, then $g' $ has a reduced expression of the form 
\[
  g' = b_1 \cdots b_s g_1'\cdots g_{k-1}' (g_k' b_s^{-1}) b_{s-1} \cdots b_1^{-1}
\]
or $g' = b_1 \cdots b_{s-1} (b_s g_1') g_2' \cdots g_k' b_s^{-1} \cdots b_1^{-1}$ depending on which among $g_1' $ or $g_k' $ is in the same factor with $b_s $. 
\item Suppose $g$ can be conjugated into one of the factors. Regardless of the chosen reduced expression of $g$, the procedure in the proof of \Cref{sec4-exist-cyclic-reduced} leads to an element contained in one of the factors. The procedure terminates at a cyclically reduced form of $g$ and the output of the procedure can't have length $>1 $ or else it is unbounded.
\end{enumerate}
\end{rmk}

Here is a lemma that will be used to prove \Cref{main-criterion}. The lemma states that if $a$ reduces the length of $g_1$ or $g_2 $ by conjugation, then it reduces the length of both $g_1$ and $g_2$.

\begin{lemmap}
\label{lemma:main-criterion}
Suppose $g_1, g_2\in G$ are such that the elements $g_1$ and $g_2$, as well as their product $g_1 g_2$, are each bounded elements. Suppose $l(g_1)>1 $. Let $h\in G$ be such that $l(h)=1$ and $l(h g_i h^{-1})<l(g_i)$ (such $h\in G$ exists by \Cref{rmk:cyc-reduced} (ii)). Then:

\begin{enumerate}[label=(\roman*)]
\item $l(h g_1 h^{-1})= l(g_1)-2 $. 
\item If $l(g_2)>1 $, then $l(hg_2 h^{-1})=l(g_2)-2 $. 
\item If $l(g_2)\leq 1 $, then $(hg_2 h^{-1})\leq 1 $. 
\end{enumerate}
\end{lemmap}

\begin{proof}
\begin{enumerate}[label=(\roman*)]
\item As in \Cref{rmk:cyc-reduced} (ii), write down a reduced expression for $g_1 $:
\[
  g_1= h_1\cdots h_r \overline{g_1} h_r^{-1}\cdots h_1^{-1} \quad ,\,  r\geq 1 .
\]
Without loss of generality, we may assume that $h_1 \in A-B $.

\par Since $l(hg_1 h^{-1})< l(g_1)$, we have $h h_1\in A\cap B $ and $h \in A-B $. Thus, the reduced expression of $h g_1 h^{-1} $ is 
\[
  h g_1 h^{-1} = (h h_1 h_2) h_3\cdots h_r \overline{g_1} h_r^{-1} \cdots h_3^{-1} (h_2^{-1} h_1^{-1} h^{-1}) \, , 
\]
which has length $l(h g_1 h^{-1})= l(g_1)-2 $. 

\item Write down a reduced expression for $g_2$ as in \Cref{rmk:cyc-reduced}(ii):
\begin{equation*}
  \begin{gathered}
    g_2=k_1\cdots k_s \overline{g_2} k_s^{-1} \cdots k_1^{-1} \quad ,\, s \geq 1. 
  \end{gathered}
\end{equation*}

\par We should have $h_1 $ and $k_1 $ belong to the same factor, i.e. $k_1\in A-B $. Suppose the contrary that $k_1\in B-A $. Then, 
\[
  g_1 g_2=h_1 \cdots h_r \overline{g_1} h_r^{-1} \cdots h_1^{-1} k_1 \cdots k_s \overline{g_2} k_s^{-1} \cdots k_1^{-1} 
\]
is reduced expression and is cyclically reduced, which implies $g_1 g_2 $ is unbounded by (i) of \Cref{rmk:cyc-reduced}. 

\par We should have $h k_1 \in A\cap B $. If not, suppose the contrary that $h k_1 \in A - B $. Therefore, $h_1^{-1} k_1= (h h_1)^{-1}\cdot h k_1 \in A -B$ too. Then, the reduced expression of $h g_1 g_2 h^{-1} $ is 

\[
  hg_1 g_2 h^{-1}= (h h_1 h_2) h_3 \cdots h_r \overline{g_1} h_{r}^{-1} \cdots h_2^{-1} (h_1^{-1} k_1) k_2 \cdots k_s \overline{g_2} k_s^{-1}\cdots k_2^{-1} (k_1^{-1} h^{-1}) .  
\]
Observe that $h h_1 h_2 \in B-A$ while $k_1^{-1} h_1^{-1}\in A-B $. Thus, $h g_1 g_2 h^{-1} $ is cyclically reduced, and by \Cref{rmk:cyc-reduced} (i), $g_1 g_2 $ is unbounded, which is contradiction.

\par In conclusion, $h g_2 h^{-1} $ has the reduced expression
\[
  h g_2 h^{-1}= (h k_1 k_2) k_3\cdots  k_s \overline{g_2} k_s^{-1}\cdots k_3^{-1} (k_2^{-1} k_1^{-1} h^{-1}),
\]
which has length $l(h g_2 h^{-1})= l(g_2)-2$. 

\item If $g_2 \in A\cap B $, then $h g_2 h^{-1} \in A$ and $l(h g_2 h^{-1})\leq 1 $. 
 
\par Suppose $l(g_2)=1$, i.e. $g_2 \not\in A\cap B $. If $g_2 \in B-A $, then 
\[
  g_1 g_2 = h_1\cdots h_r \overline{g_1} h_r^{-1} \cdots h_1^{-1} g_2 
\]
is reduced expression which is cyclically reduced, which is unbounded by \Cref{rmk:cyc-reduced} (i). It contradicts the assumption that $g_1 g_2$ is bounded.

\par Thus, we should have $g_2 \in A - B$. Hence, $hg_2 h^{-1}\in A $ and $l(hg_2 h^{-1})\leq 1 $, as desired. 
\end{enumerate}
\end{proof}

\begin{proof}[Proof of \Cref{main-criterion}]
One direction of (i) is obvious. Thus, we start from the assumption that $g_1$ has the maximal length among the generators, and that both the generators $g_1, \cdots, g_n$ and the products $g_{1} g_2$, $g_1 g_3$, $\cdots, g_{1} g_n$ are bounded elements. Our goal is to prove that $\langle g_1,\cdots, g_n \rangle $ is conjugate to a subgroup of either $A$ or $B$, and that the conjugator is as described in the statement. 

If $l(g_1)=0 $, then since $g_{1}$ has maximal length, it follows
that $l(g_j)=0 $ for all $j=1,\cdots, n $. Therefore, $\langle g_1,\cdots, g_n\rangle \subset A\cap B $, and we are done. 

Suppose $l(g_{1})=1 $. If $g_{1}\in A -B $, the assumption that $
g_{1} g_2, g_{1} g_3,\cdots, g_{1} g_n $ are bounded elements implies that $g_1,\cdots, g_n\in A $. Thus, $\langle g_1,\cdots, g_n \rangle \subset A $, and we are done. In this case, the conjugator is just the identity. 

For the general case $l(g_1) \geq 1 $, we proceed by induction on
the maximal length among the generators. The base case when the maximal
length is 1 is already proven. Suppose $l(g_1)>1 $. The induction
hypothesis states that (i) and (ii) is true for sets of generators whose
maximal length among the generators is strictly smaller than $l(g_1)$. 

Since $g_1$ is bounded, it admits a reduced expression of the form 
\[
  g_1= c_1\cdots c_s \overline{g_1} c_s^{-1} \cdots c_1^{-1}
\]
as in the \Cref{rmk:cyc-reduced}(ii). Thus, $l(g_1)=2s+1 >1$ is odd, and any reduced expression $g_1=k_1\cdots k_{2s+1} $ has the same length $l(g_1)=2s+1 $ (\Cref{reduced-expr-unique}). 

Since the $k_1$ and $k_{2s+1} $ are in the same factor, conjugating $g_1 $ by $k_1^{-1} $ decreases the length, i.e. $l(k_1^{-1} g_1 k_1)<l(g_1) $. By \Cref{lemma:main-criterion}(i), we have precisely $l(k_1^{-1} g_1 k_1)= l(g_1)-2 $. Thus, the reduced expression of $k_1^{-1} g_1 k_1 $ should be
\[
  k_1^{-1} g_{1} k_1 = k_2 \cdots k_{2s-1} (k_{2s} k_{2s+1} k_1^{-1}) \, 
\] 
(if $l(g_1)=3 $, then $k_1^{-1} g_1 k_1= k_2 k_3 k_1^{-1} $ which has length 1).

\par Applying \Cref{lemma:main-criterion} to each pair $g_{1} g_2$, $g_{1} g_3, \cdots, g_{1} g_n  $, we deduce that 
\[ l(k_1^{-1} g_j k_1) \leq l(k_1^{-1} g_{1} k_1) < l(g_{1}) \text{ for all } j=1,\cdots, n . \]
Thus, the conjugated generators $k_1^{-1} g_1 k_1, \cdots, k_1^{-1} g_n k_1 $ have strictly smaller maximal length, and $k_1^{-1} g_{1} k_1 $ has the maximal length among the conjugated generators. We then apply the induction hypothesis to conclude, i.e. $k_1^{-1} H k_1 $ is conjugate into one of the factors with the conjugator $(k_2 \cdots k_s)^{-1} $.

\end{proof}

\section{$\Aut(\AAA^2) $ as an Amalgamated Product of $\Aff $ and $\JJ $}
\label{section:conjugacy-aut-of-affine}

In this section, I use the results of \Cref{section:grp-theory} to give the criterion for when a finitely generated subgroup $G\subset \Aut(\AAA^2) $ is conjugate to either $\Aff $ or $\JJ $ and compute the conjugator. The treatment of basic facts about the polynomial automorphism group follows the treatment of \cite{Friedland-Milnor}, \cite{Essen} and \cite{furter-iteration}.

\par Following is a well-known theorem in the theory of polynomial automorphisms:

\begin{theo}[Jung-van der Kulk \cite{Jung1942} \cite{vdK1953}]
$\Aut(\AAA^2) $ is generated by $\Aff$ and $\JJ $. 
\end{theo}

As in Blanc-Stampfli's Theorem, the theory of links provides swift proof of the above theorem \cite[Corollary 2.13]{blanc-stampfli}. Upon computation of the degrees, it becomes evident that there is no nontrivial relation between $\Aff $ and $\JJ $. Hence, $\Aut(\AAA^2)= \Aff *_{\Aff \cap \JJ} \JJ $.

\begin{lemmap}
\label{lemma:len-deg}
\cite[Lemma 5.1.2]{Essen}
Suppose $\phi=\beta_l \alpha_l \cdots \beta_1 \alpha_1 \in \Aut(\AAA^2)$ is such that $\beta_i\in \JJ \setminus \Aff$ for $1\leq i \leq l $, $\alpha_i \in \Aff \setminus \JJ$ for $2\leq i \leq l $ and $\alpha_1 \in \Aff $. Then $\bideg{\phi}= (\prod_{i=1}^l \deg{\beta_i}, \prod_{i=1}^{l-1} \deg{\beta_i}  ) $. 
\end{lemmap}
\begin{proof}
Case when $l=1$ is obvious. Proceed by induction and assume the assertion is true for $l-1 $, hence that $\bideg{(\beta_{l-1}\alpha_{l-1}\cdots \alpha_1 \beta_1)}= (\prod_{i=1}^{l-1}\deg{\beta_i}, \prod_{i=1}^{l-2} \deg{\beta_i}) $. Since $\alpha_{l}\in \Aff\setminus \JJ $ is not upper triangular, the $y$-component of $\alpha_{l}\beta_{l-1}\alpha_{l-1}\cdots \alpha_1\beta_1 $ has degree $\prod_{i=1}^{l-1}\deg{\beta_i} $ and the $x$-component has degree $\leq \prod_{i=1}^{l-1} \deg{\beta_i} $. Thus the bidegree of $\phi=\beta_l \alpha_l \cdots \beta_1 \alpha_1 $ is $\bideg{\phi}= (\prod_{i=1}^l \deg{\beta_i}, \prod_{i=1}^{l-1} \deg{\beta_i}) $. 
\end{proof}

\begin{coro}
\label{coro:aut-amalgam}
\cite[Corollary 5.1.6]{Essen}
Let $\phi=(\phi_1, \phi_2) \in \Aut(\AAA^2) $ be such that $\bideg{\phi}= (d_1, d_2) $. 
\begin{enumerate}[label=(\roman*)]
  \item $\Aut(\AAA^2)$ is amalgamated product of $\Aff $ and $\JJ $ along their intersection and the intersection is the group of upper triangular affine linear transformations.
  \item $d_1 \mid d_2 $ or $d_2 \mid d_1 $. 
  \item The first letter in reduced expression of $\phi $ is in $\JJ $ if and only if $d_1 > d_2 $
  \item The first letter in reduced expression of $\phi $ is in $\Aff $ if and only if $d_1 \leq d_2 $
\end{enumerate}
(If $\phi=g_1\cdots g_r $ is a reduced expression, then the first letter of $\phi $ is $g_1 $.)
\end{coro}

\Cref{coro:aut-amalgam} provides an algorithm to compute a reduced expression for elements in $\Aut(\AAA^2) $. 

\begin{propp}
\label{prop:factorization-algo}
\cite[88]{Essen}
\begin{enumerate}[label=(\roman*)]
\item The set $\{(x, ax+y): a\in \kk \} \cup \{(y,x) \} $ is coset representatives of $\Aff / (\Aff \cap \JJ) $. The set $\{(x+p(y),y): \text{p(y) has no linear or constant terms} \} $ is coset representatives of $\JJ / (\Aff \cap \JJ) $.
\item There exists an algorithm to compute the normal form of an automorphism $\phi \in \Aut(\AAA^2) $ with respect to the coset representatives of (i).
\end{enumerate}
\end{propp}

Following criterion is brought from \cite{furter-iteration}, \cite{Friedland-Milnor}. For the convenience of the reader, I present its proof using the formalism of cyclically reduced form. We say that a subgroup of $\Aut(\AAA^2) $ \textit{is of bounded degree} if the degree of its elements is bounded.  

\begin{propp}
\label{prop:bdd-len-criterion}
\cite[Proposition 1.5]{furter-iteration}
Let $\phi \in \Aut(\AAA^2) $. Then the following are equivalent:
\begin{enumerate}[label=(\roman*)]
  \item $\phi $ is conjugate to an element of $\Aff $ or $\JJ $. 
  \item $\langle \phi \rangle $ is subgroup of bounded length. 
  \item $\langle \phi \rangle $ is subgroup of bounded degree.
  \item $\deg{\phi^{2}}\leq \deg{\phi} $
  \item For all integer $n$, $\deg{\phi^{n}}\leq \deg{\phi} $
\end{enumerate}
\end{propp}
\begin{proof}
The equivalence (i)$\iff $(ii) is Serre's Theorem (\Cref{serre-tree}). 

\par Note that a conjugate of a subgroup of bounded degree is of bounded degree by the computation of \Cref{lemma:len-deg}. Indeed, if $H$ is a subgroup of bounded degree, where its elements have a degree less than some $M>0 $, then the degree of elements in a conjugate $gHg^{-1} $ is less than $l(g)^2 M $.  
\\(ii)$\implies$(iii): By Serre's Theorem (\Cref{serre-tree}), that $\langle \phi \rangle$ is of bounded length implies that $\langle\phi \rangle $ is conjugate to a principal subgroup $\langle \overline{\phi}\rangle$ which is contained in $\Aff $ or $\JJ $. Such the principal subgroup is obviously of bounded degree, which implies that $\langle \phi \rangle$ is of bounded degree too. 
\\(iii)$\implies$(ii): Suppose $\langle \phi \rangle$ is not of bounded length. Then, by \Cref{sec4-unique-cyc-reduced}, $\phi $ is conjugate to a cyclically reduced form of even length. Let $\phi_1 \cdots \phi_{2k} $ ($k \geq 1 $) be a reduced expression of the cyclically reduced form of $\phi$. By the computation of \Cref{lemma:len-deg}, the degree of the elements in $\langle \phi_1 \cdots \phi_{2k} \rangle $ is unbounded. Thus, the degree of the elements in $\langle \phi \rangle$ is unbounded.

Now, we prove the implications in the order (i)$\implies$(v)$\implies$(iv)$\implies$(i).
\\ (i)$\implies $(v): Let $\phi= g_1\cdots g_l \overline{\phi} g_l^{-1} \cdots g_1^{-1} $ be a reduced expression of $\phi$ (\Cref{rmk:cyc-reduced} (ii)). Then, $\phi^{n}= g_1\cdots g_l \overline{\phi}^n g_l^{-1}\cdots g_1^{-1} $ and $l(\phi^{n})\leq l(\phi) $. Moreover, some letters around $\overline{\phi}^n $ might merge together upon taking the reduced expression. Using the computation of \Cref{lemma:len-deg}, we have $\deg{\phi^n}\leq \deg{\phi} $ for all $n\geq 1$. 
\\ (v)$\implies $(iv): This is obvious.
\\ (iv)$\implies$(i): I prove the contrapositive. Suppose $\phi $ is not conjugate with any elements in $\Aff \cup \JJ $. Write down a reduced expression of $\phi $ as in \Cref{rmk:cyc-reduced} (ii), 
\[
  \begin{gathered}
  \phi= g_1\cdots g_r \phi_1\cdots \phi_{2k-1} (\phi_{2k} g_r^{-1}) g_{r-1}^{-1}\cdots g_1^{-1} \quad \text{, or } \\
  \phi= g_1 \cdots g_{r-1} (g_r \phi_1) \phi_2 \cdots \phi_{2k} g_r^{-1} \cdots g_1^{-1} 
  \end{gathered}
\]
where $\phi_1\cdots \phi_{2k} $ is the cyclically reduced form of $\phi $. Then, the reduced expression of $\phi^2 $ is 
\[
  \begin{gathered}
  \phi^2= g_1\cdots g_r \phi_1 \cdots \phi_{2k} \phi_1 \cdots \phi_{2k-1} (\phi_{2k} g_r^{-1}) g_{r-1}^{-1} \cdots g_1^{-1} \quad \text{, or } \\
  \phi^2= g_1 \cdots g_{r-1} (g_r \phi_1) \phi_2 \cdots \phi_{2k} \phi_1\cdots \phi_{2k} g_r^{-1} \cdots g_1^{-1} 
  \end{gathered}
\]
Compare the degree of $\phi^2 $ with that of $\phi$ using \Cref{lemma:len-deg}. Observe that all letters from $\JJ $ in the reduced expression of $\phi$ are also present in $\phi^2 $. However, the letters in the cyclically reduced form appear as duplicate in the reduced expression of $\phi^2 $. Thus, $\deg{\phi^2}> \deg{\phi} $. 
\end{proof}

By above proposition, we can simply say that an element $\phi \in \Aut(\AAA^2) $ is \textit{bounded} without distinguishing whether it refers to the degree or length; it means that the subgroup $\langle \phi \rangle$ is bounded both in degree and length.

Now, we combine the conjugacy criterions \Cref{main-criterion} and \Cref{prop:bdd-len-criterion} with the factorization algorithm \Cref{prop:factorization-algo} to provide explicit solution to the conjugacy problem in $\Aut(\AAA^2) $. 

\begin{theop}
\label{theo:conj-criterion-aut}
Let $H\leq \Aut(\AAA^2) $ be a finitely generated subgroup. Then, there exists the algorithm to decide whether $H $ is conjugate into $\Aff $ or $\JJ $ and compute the conjugator if so. 
\end{theop}
\begin{proof}
We describe the algorithm to do this. Let $H=\langle g_1,\cdots, g_n \rangle \leq \Aut(\AAA^2) $ be given. If $l(g_i)\leq 1$ for all $i=1,\cdots n $, then the answer to the conjugacy problem is immediate; $H$ is conjugate into one of the factors if and only if all the generators belong to the same factor, and the conjugator is the identity.

Suppose at least one of the generators is not of length $\leq 1 $. Compute reduced expressions and the lengths of the generators $g_1, \cdots, g_n $ using \Cref{prop:factorization-algo}. After reindexing, suppose $g_{1} $ has maximal length among the generators. Check whether the elements $g_1,\cdots, g_n $ and the products $g_{1} g_2$ $g_1 g_3, \cdots, g_{1} g_n $ are bounded elements. If one of them is not bounded element, then, by \Cref{main-criterion}, $H$ is not conjugate into one of the factors and the algorithm terminates.

\par Suppose the elements $g_1,\cdots, g_n $ and the products $g_{1} g_2$, $g_1, g_3 \cdots, g_{1} g_n $ are bounded. Then, by \Cref{main-criterion}, $H$ is conjugate into one of the factors. If $g_{1}=k_1\cdots k_{2s+1} $ is a reduced expression of $g_{1} $, then the conjugator is $(k_1\cdots k_s)^{-1} $ by \Cref{main-criterion}. 
\end{proof}

\section{Computation of the Curves Invariant by $\phi \in \JJ $}
\label{section:invar-curves}

The computations of irreducible invariant curves that can occur was already demonstrated in \cite[Section 4]{blanc-stampfli}. In this section, we compute the totality of the invariant subvarieties (not necessarily irreducible) in $\AAA^2 $ which are invariant under a given $\phi \in \JJ $, and our demonstration is aimed towards developing the orbit Zariski closure computation algorithm. 

To compute the orbit Zariski closure with respect to a finitely generated group, we need means to derive the invariant of several maps from the information of the invariant of each map. It is done by obtaining the totality of the invariant subvarieties for given $\phi \in \JJ $, then interwining the maps using the equivariance criterion (\Cref{lemma:equivar-affine}, \Cref{lemma:equivar-projective}).

The byproduct of the computations is the case-by-case proof of the fact that all torsion elements in $\Aut(\AAA^2) $ are affine linear transform after a coordinate change. In particular, we precisely identify which elements are the torsion elements in $\Aut(\AAA^2) $.

\par In this section, we assume that the base field $\kk$ is algebraically closed and characteristic 0. While our focus in this paper remains on characteristic 0, I am sure that the argument could be readily extended to positive characteristic with little effort.

For the clarity of the presentation, we introduce the concept of \textit{invariant subvariety lattice}. In this paper, \textit{a variety} may not be an irreducible scheme. Check the notations and conventions in the introduction (\Cref{Section:Introduction}). 

\begin{definition}
  \label{def:invar-subvar-lattice}
  Let $G$ be a group acting algebraically on a variety $X$. The \textit{invariant subvariety lattice $\mathfrak{I}$} of $X$ by $G$ is 
  \[
    \begin{gathered}
    \frakI_{G}:= \{ Y\subset X: \text{$Y$ is $G$-invariant closed subvariety of $X$} \} \, , \\
    \frakI_{G,x}:= \{Y\subset X: \text{$Y$ is $G$-invariant closed subvariety of $X$ and $x\in Y$} \} \, , \\ 
    \frakI_{G, Z}:= \{Y \subset X: \text{$Y $ is $G$-invariant closed subvariety of $X$ and $ Z\subset Y$} \} \, . 
    \end{gathered}
  \]
  For a lattice $\frakI $ of subvarieties in $X$, the \textit{support of the lattice $\frakI $} is 
  \[
    \Supp{\frakI}:= \bigcup_{\substack{Y\in \frakI_{G,x} \\ Y\neq X }} Y \, . 
  \]

  For a lattice $\mathfrak{I}$ of subvarieties of $X$, if a subset $\mathcal{I} \subset \mathfrak{I}$ is such that every element in $\mathfrak{I}$, except $\AAA^2 $, is a union of elements in $\mathcal{I}$, then we say that \textit{$\mathcal{I}$ generates $\mathfrak{I}$}. 
\end{definition}

\begin{rmk}
\label{rmk:invar-subvar-lattice}
\begin{enumerate}[label=(\roman*)]
  \item Thus defined $\frakI_{G,x} $ is indeed a lattice since $\frakI_{G,x} $ is closed under finite intersections and finite unions. The least element of $\frakI_{G,x}$ is the orbit Zariski closure of $x$ by $G$. The greatest element of $\frakI_{G,x} $ is the whole $X$.
  
  \par The invariant subvariety lattice $\frakI_{G} $ is also a lattice with the least element as the empty set $\emptyset $ and the largest element $X $. 
  \item The 0-dimensional elements in the $\frakI_{G} $ are precisely the torsion points of $G$.
  \item If $G=\langle g_1,\cdots, g_r \rangle $, then $\frakI_{G,x}$ is the intersection of lattices:
  \[
    \frakI_{G,x}= \bigcap_{i=1}^r \frakI_{g_i,x} \, . 
  \]
\end{enumerate}
\end{rmk}

\par The goal of this section is to compute $\frakI_{\phi} $ when $\phi\in \JJ$. We do so by explicitly describing the set of generators to the lattice $\frakI_{\phi} $, which consists of fibers to a $\phi $-equivariant map, and a curve transversal to the fibers. The computations of the invariant subvariety lattices $\frakI_{\phi} $ is then used in \Cref{section:computation-orbit-closure} to compute the orbit Zariski closures for finitely generated groups in $\Aut(\AAA^2) $. 

Before going into the computations, we first clarify what it means to be equivariant. To deal with several automorphisms and their invariant subvariety lattice, we need the following lemmas about the equivariance. The lemma says that the global equivariance is equivalent to equivariance on just one point, and that the set-theoretic equivariance is equivalent to algebraic equivariance. 

\begin{lemmap}
\label{lemma:equivar-affine}
Let $\pi \in \kk[x,y] $ be a nonconstant polynomial such that general fibers are irreducible and reduced thought as a morphism $\pi:\AAA^2 \rightarrow \AAA^1 $. Let $\phi \in \Aut(\AAA^2) $ be an automorphism. Then, the following are equivalent (the fibers are scheme-theoretic):
\begin{enumerate}[label=(\roman*)]
  \item There exists an automorphism $\overline{\phi}\in \Aut(\AAA^1) $ such that $\pi\circ \phi = \overline{\phi}\circ \pi $. 
  \item $\phi^*\pi \in \kk[\pi] $ is a degree 1 polynomial in $\pi $. 
  \item For all $p\in \AAA^2 $, it holds $\phi(\pi^{-1}(\pi(p))) = \pi^{-1}(\pi(\phi(p))) $.
  \item There exists $p\in \AAA^2 $ such  that $\pi^{-1}(\pi(p))$ is irreducible and reduced, and $\phi(\pi^{-1}(\pi(p)))= \pi^{-1}(\pi(\phi(p))) $.
\end{enumerate}
\end{lemmap}
\begin{proof}
(i)$\implies$(ii): The condition (i) says that $\phi^* \pi=\overline{\phi}(\pi) \in \kk[\pi] $ where $\overline{\phi} \in \Aut(\AAA^1)$ is thought as a degree 1 polynomial in one variable. 
\\ (ii) $\implies$ (i): The condition (ii) says that $\phi^* $ restricts to $\kk $-algebra automorphism of $\kk[\pi]\subset \kk[x,y] $. The $\kk$-algebra automorphism gives the automorphism of $\Spec{\kk[\pi]}\cong \AAA^1 $ which commutes as in the condition (i).
\\ (i)$\implies$ (iii) : This is set-theoretically obvious.
\\ (iii)$\implies$ (iv) : Obvious. 
\\ (iv)$\implies$ (ii): Denote as $\pi(p)=t_1, \pi(\phi(p))=t_2 \in \AAA^1 $. Then, the condition (iv) says that $\phi$ restricts to isomorphism of curves $\phi:V(\pi - t_1)\cong V(\pi - t_2) \subset \AAA^2$. Hence, $\phi^* \pi - t_1 $ is generator of the principal ideal $(\pi - t_2) \subset \kk[x,y] $, and there exists a constant $s\in \kk^{\times}$ such that $\phi^* \pi -t_1= s(\pi - t_2) $, which is the desired conclusion. 
\end{proof}

We need the version of above lemma for a rational map into projective line. In this case, we need to two equivariant points compared to the affine case.

\begin{lemmap}
\label{lemma:equivar-projective}
Let $\pi\in \kk(x,y)$ be a rational function such that general fibers are irreducible and reduced thought as a rational map $\pi:\AAA^2 \dashrightarrow \PP^1 $. Let $\phi\in \Aut(\AAA^2) $ be an automorphism such that the domain of definition $\dom(\pi)$ is $\phi$-invariant. Then, the following are equivalent ($\pi^{-1}(p)$ is the strict transform of the rational map):
\begin{enumerate}[label=(\roman*)]
  \item There exists a projective automorphism $\overline{\phi} \in \Aut(\AAA^2)$ such that $\pi\circ \phi= \overline{\phi}\circ \pi $. 
  \item $\phi^*\pi \in \kk(\pi) $ is a linear fractional transform of $\pi $.
  \item For all $p\in \dom(\pi)$, it holds $\phi(\pi^{-1}(\pi(p))) = \pi^{-1}(\pi(\phi(p)))$.
  \item There exists $p$, $q\in \dom(\pi) $ with $\pi(p)\neq \pi(q) $ such that $\pi^{-1}(\pi(p)) $ and $\pi^{-1}(\pi(q)) $ are irreducible and reduced, $\phi(\pi^{-1}(\pi(p)))= \pi^{-1}(\pi(\phi(p))) $, and $\phi(\pi^{-1}(\pi(q)))= \pi^{-1}(\pi(\phi(q))) $
\end{enumerate}
\end{lemmap}
\begin{proof}
The equivalence of (i) and (ii) is the same as in \Cref{lemma:equivar-affine}. 
\par The implications (i)$\implies$(iii)$\implies$(iv) are obvious. Let us prove the implication (iv)$\implies $(ii). 

\par Denote as $\pi=\frac{f}{g}\in \kk(x,y) $ where $f, g\in \kk[x,y]$ are relative prime polynomials, which is identified with the rational map into projective line $[f:g]:\AAA^2 \dashrightarrow \PP^1 $. Denote as $\pi(p)=[t_1:t_2]$ and $\pi(q)=[t_3:t_4]\in \PP^1$. Denote as $\pi(\phi(p))=[s_1:s_2]$ and $\pi(\phi(q))=[s_3:s_4]\in \PP^1 $. The condition that $\phi(\pi^{-1}(\pi(p)))= \pi^{-1}(\pi(\phi(p))) $, and $\phi(\pi^{-1}(\pi(q)))= \pi^{-1}(\pi(\phi(q))) $ is stated in terms of ideals in $\kk[x,y] $ as 
\[
  (t_2 f + t_1 g) = (s_2 \phi^* f+ s_1 \phi^* g) \, , \quad (t_4 f+ t_3 g) = (s_4 \phi^*f + s_3 \phi^* g ) \, . 
\]
Since $[t_1:t_2]\neq [t_3:t_4] $ by the assumption, above equalities of principal ideals imply the equality of the vector spaces $\kk\langle f,g \rangle= \kk\langle \phi^*f, \phi^* g\rangle $. Hence, there exists scalars $a_1, a_2, a_3, a_4\in \kk $ such that 
\[
\phi^*f = a_1 f + a_2 g \, , \quad \phi^*g= a_3 f+ a_4 g \, . 
\]
Thus, we arrive at the desired conclusion as 
\[
\phi^*\pi = \frac{a_1\pi +a_2 }{a_3 \pi + a_4 } \in \kk(\pi) \, .
\]

\end{proof}

If $G \subset \Aut(\AAA^2) $ and $f\in \kk(x,y) $ satisfy the conditions of \Cref{lemma:equivar-affine} or \Cref{lemma:equivar-projective}, then we will say that $f$ is \textit{$G$-equivariant}.

\begin{notation}
\label{Notation:dual-rep-of-jonquieres}
Let $\JJ_n $ act on the affine plane $\AAA^2=\Spec{\kk[x,y]} $ as the polynomial map on the coordinates. Hence, we have the dual representation of $\JJ_n $ on $\kk[x,y]$. Throughout the paper, let us fix and clarify the direction from which $\JJ_n $ acts on $\kk[x,y]$ as dual representation. The de Jonqui\`eres group $\JJ_n $ acts on $\kk[x,y] $ from the \textit{right}, i.e. 
\begin{equation*}
  (f.\phi)(x)= f(\phi(x)), \,\,\text{or equivalently } f.\phi= \phi^* f \,\, \text{ for } f\in \kk[x,y],\, \phi\in \JJ_n \, .
\end{equation*}
  
\par Throughout the paper, we fix the notation for $\JJ_n $-invariant subspace 
\[
  V_n:= \kk \langle 1, y, y^2, \cdots, y^n \rangle \subset \kk[x,y] \, , 
\]
where the angular brackets mean the vector space generators. 

\par Thus, $\phi^* $ is nonsingular endomorphism of the vector space $V_n $. In the matrix notation, the action of $\phi=(ax + a_n y^n + a_{n-1}y^{n-1}+\cdots + a_0, by+c )\in \JJ_n $ on $f=d'x+ d_ny^n + \cdots + d_0 \in V_n $ is written as:
\begin{equation}
  \label{eq:matrix-rep-of-jonquieres}
  \phi^*f=  \begin{pmatrix}
      1 & y & y^2 & \cdots & y^n & x
    \end{pmatrix}
    \begin{pmatrix}
      1 & c & c^2 & \cdots & c^n & a_0 \\
       & b & 2bc &  & nbc^{n-1} & a_1 \\
       &  & b^2 &  & \vdots &   \\
       &  &  & \ddots & nb^{n-1}c & a_{n-1} \\
       & \text{\huge0} &  &  & b^n & a_n \\
       &  &  &  &   & a
      \end{pmatrix}
      \begin{pmatrix}
        d_0 \\
        d_1 \\
        \vdots \\
        d_n \\
        d'
        \end{pmatrix}  \, . 
\end{equation}
Thus, the matrix representation defines an \textit{anti-homomorphism} from $\JJ_n $ to $\GL_{n+2} $ (because $\JJ$ acts on $V_n $ from right). Throughout the paper, the \textit{matrix form of $\phi \in \JJ_n $} will be as above. 
\end{notation}

\begin{rmk}
Although we are computing only for $\phi \in \JJ_n $, this eventually computes the invariant curves for $\Aff $ too. This is because every element in $\Aff $ can be conjugated to an element of $\JJ_1= \JJ\cap \Aff $. Namely, an element $\psi=(a_1 x+b_1 y+c_1, a_2 x+ b_2 y+ c_2)\in \Aff $ has matrix representation 
\[
  \psi^*f= \begin{pmatrix}
    1 & y & x
    \end{pmatrix}
    \begin{pmatrix}
      1 & c_2 & c_1 \\
      0 & b_2 & b_1 \\
      0 & a_2 & a_1
      \end{pmatrix}
      \begin{pmatrix}
        d_1 \\
        d_2 \\
        d_3
        \end{pmatrix}
\]
where $f=d_1+ d_2y+ d_3x \in \kk\langle 1, y, x\rangle $. By finding the Jordan canonical form of the lower right $2\times 2 $ matrix, we can easily conjugate $\psi $ to an element of $\JJ\cap \Aff $. 
\end{rmk}

Following is an elementary diagonalization lemma. 

\begin{lemmap}
\label{lemma:diagonalization}
Let $\phi=(ax+P(y), by+c)\in \JJ_n $. Suppose that $b\neq 1 $, or $(b,c)=(1,0) $. Then, there exists a change of coordinate $\alpha=(\tilde{x}, \tilde{y})\in \Aut(\AAA^2) $, and a polynomial $h_1(y)=\tilde{h}_1(\tilde{y}) \in \kk[y]= \kk[\tilde{y}] $ such that $\tilde{\phi}=\alpha\circ \phi \circ \alpha^{-1}$ is
\[
  \tilde{\phi}^n(\tilde{x}, \tilde{y})= (a^n \tilde{x}+ na^{n-1} \tilde{h}_1(\tilde{y}), b^n \tilde{y}) \, \text{ for all } n\in \ZZ \, . 
\]
Let $P(y)=\tilde{P}(\tilde{y})= \sum_{i=0}^n \tilde{a}_i \tilde{y}^i $. Then, the polynomials $\tilde{x}$, $\tilde{y} $, and $h_1(y)=\tilde{h}_1(\tilde{y})\in \kk[y] $ are:
\begin{gather*}
  \tilde{y}= \begin{cases}
    y & , \text{ if } (b,c)=(1,0) \\ 
    y-\frac{c}{1-b} & , \text{ if } b\neq 1
  \end{cases} \\ 
  \tilde{x}= x- h(y), \quad \text{where } h(y)=\tilde{h}(\tilde{y}):= \sum_{\substack{i=0 \\ a \neq b^i}}^n \frac{\tilde{a}_i}{b^i-a}\tilde{y}^i \in \kk[y]= \kk[\tilde{y}] \\
  h_1(y)=\tilde{h}_1(\tilde{y}):= (\phi^*-a)\tilde{x}= \sum_{\substack{0 \leq i \leq n\\ a=b^i}} \tilde{a}_i \tilde{y}^i  \in \kk[y]= \kk[\tilde{y}] \, . 
\end{gather*}
\end{lemmap}
\begin{proof}
The lemma is the computation of the Jordan canonical form of the vector space endomorphism $\phi^*\in \End_{\kk}(V_{n}) $. The polynomials $\tilde{x} $, $\tilde{y} $, and $\tilde{h}_1(\tilde{y}) $ as in the lemma satisfy
\[
  \begin{gathered}
    \phi^* \tilde{x}= a \tilde{x}+ \tilde{h}_1(\tilde{y})\, , \\ 
    \phi^* \tilde{y}= b\tilde{y} \, , \\
    (\phi^*-a)^2 \tilde{x}= (\phi^*-a) \tilde{h}_1(\tilde{y})= 0 \, . 
  \end{gathered}
\]
Then, from above relations, it is immediate to check that 
\[
  \begin{gathered}
    (\phi^*)^n\tilde{x}= a^n \tilde{x}+ na^{n-1} \tilde{h}_1(\tilde{y})  \, , \\
    (\phi^*)^n \tilde{y}= b^n \tilde{y} \, ,
  \end{gathered}
\]
hence the lemma.
\end{proof}

The following lemma is used to compute Zariski closure of a subgroup generated by a single element $\phi \in \GL_2(\kk) $. The transcendence of exponential function is a very classical result, but I couldn't find the relevant proof in the particular case when the domain is restricted to $\ZZ $. Thus, we exhibit the elementary proof of the following.

\begin{lemmap}
\label{lemma:transcendence}
Suppose that $a,b\in \kk^{\times} $ are not roots of unity.
\begin{enumerate}[label=(\roman*)]
  \item Suppose that there exist no nonzero integers $r_1 $ and $r_2 $ such that $a^{r_1}=b^{r_2} $. Then, the subset $\{(a^n, b^n)\, \vert \, n \in \NN \} \subset \AAA^2$ is Zariski dense.  
  \item The subset $\{(n, a^n)\, \vert \, n \in \NN \} \subset \AAA^2$ is Zariski dense. 
\end{enumerate}
\end{lemmap}
If $\kk=\CC $, then the condition of (i) is simply stated that $\log{a} $ and $\log{b} $ are linearly independent over $\QQ $.  

\begin{proof}
\begin{enumerate}[label=(\roman*)]
\item Suppose there exists a polynomial $f(x,y) \in \kk[x,y]$ such that $f(a^n, b^n)=0 $ for all $ n\in \NN$. Then, $f(a^n, b^n)=0$ gives the linear dependence relation between the monoid characters $n\mapsto (a^ib^j)^n $ where $i,j $ runs through the monomials $x^i y^j $ of polynomial $f(x,y)$. The hypothesis of the statement guarantees that $n\mapsto (a^i b^j)^n $ are all distinct characters for each monomial $x^i y^j$ of $f$. Such the linear dependence relation contradicts the well known classical result of the linear independence of monoid characters. Thus, such the polynomial $f$ doesn't exist. 
\item Suppose there exists a polynomial $f(x,y) \in \kk[x,y] $ such that $f(n, a^n)=0$ for all $n\in \NN$. Write as 
\[
f(x,y)= b_{d_2}(x)y^{d_2}+ b_{d_2-1}(x) y^{d_2-1}+\cdots b_1(x)y+b_0(x)
\]
where $d_2$ is the $y$-degree of $f$, and $\deg{b_{d_2}(x)}=d_1 $. The proof is done by induction on the $y$-degree $d_2 $, then secondary induction on the degree $d_1 $ of the leading coefficient $b_{d_2}(x) $. 

The base case $d_2=0 $ is trivial. The base case $d_2=0 $ asks whether a polynomial $f(x)\in \kk[x,y] $ such that $f(n)=0 $ for all $n\in \NN $ is trivial. Such the polynomial is obviously $f=0 $. 

\par We proceed inductively, and suppose we have proved that any polynomial $f(x,y)$ such that $\deg_y{f}<d_2 $ and $f(n,a^n)=0 $ for all $n\in \NN$ is trivial. We claim that there does not exist nontrivial $f(x,y)\in \kk[x,y] $ such that $\deg_y{f}=d_2 $, $\deg{b_{d_2}(x)}=0 $, and $f(n,a^n)=0 $ for all $n\in \NN$. 

After dividing by the constant leading coefficient and rewriting the coefficients, 
\[
f(x,y)= y^{d_2}+\sum_{j=0}^{d_2-1} b_j(x) y^j \, . 
\]
The condition implies that
\[
a^{d_2} f(n, a^n)- f(n+1, a^{n+1})= \sum_{j=0}^{d_2-1} (a^{d_2}b_j(n)- a^j b_j(n+1) ) a^{jn}=0 \quad , \quad \text{ for all }n\in \NN \, . 
\]
Then, the nontrivial polynomial $\sum_{j=0}^{d_2-1} (a^{d_2}b_j(x)- a^j b_j(x+1) )y^j $ vanishes on the set $\{(n,a^n) \, \vert \, n \in \NN \} $, and has $y$-degree less than $d_2 $, contradicting the induction hypothesis. Thus, we have proved the claim for $\deg{b_{d_2}(x)}= d_1=0$. 

Proceed inductively with respect to the $\deg{b_{d_2}(x)}=d_1 $, and suppose we have proved that there does not exist nontrivial polynomial $f$ such that $\deg_y{f}=d_2 $, $\deg{b_{d_2}(x)}<d_1 $ and $f(n, a^n)=0 $ for all $n\in \NN $. We claim that there does not exists nontrivial polynomial $f$ with $\deg_y{f}=d_2 $, $\deg{b_{d_2}(x)}=d_1 $ and $f(n,a^n)=0 $ for all $n\in \NN $. 

\par Existence of such the polynomial gives a nontrivial relation 
\[
\begin{aligned}
  a^{d_2} f(n, a^n)- f(n+1, a^{n+1}) = a^{d_2} (b_{d_2}(n)- b_{d_2}(n+1)) a^{d_2 n} + \sum_{j=0}^{d_2-1} (a^{d_2} b_j(n)- a^j b_j(n+1)) a^{jn}=0
\end{aligned}
\]
for all $n\in \NN $. Since the leading coefficient of the polynomial 
\[
a^{d_2}(b_{d_2}(x)- b_{d_2}(x+1)) y^{d_2}+\sum_{j=0}^{d_2-1}(a^{d_2} b_j(x)- a^j b_j(x+1)) y^j
\]
has degree strictly less than $d_1$, we reach the contradiction against the induction hypothesis. Thus, we have proved the claim. Conclude by induction. 
\end{enumerate}
\end{proof}

\par Now, we compute the invariant subvariety lattices for a single $\phi \in \JJ $. The upshot is that the invariant subvarieties of a single automorphism of $\AAA^2 $ falls into one of the three classes described in \Cref{def:invar-subvar-lattice-classification}. 

\begin{propp}
Let $\phi=(ax+ P(y), by+c)\in \JJ_n $ be such that $\deg{P(y)}=n $. Suppose $b=1 $ and $c\neq 0 $. Then, the following are true:
\begin{enumerate}[label=(\roman*)]
\label{prop:computations-b-is-one}
  \item Suppose $a=1$. Then, $\Supp{\frakI_{\phi}}=\AAA^2 $, and $\frakI_{\phi} $ is generated by the fibers of the quotient map $\hat{x}:\AAA^2 \rightarrow \AAA^1:(x,y) \mapsto x-g(y) $ where $\deg{g}=n+1 $. In particular, $\frakI_{\phi} $ contains no 0-dimensional element, i.e. $\phi $ has no torsion point. 

  \item Suppose $a \neq 1$ is root of unity. Then, $\Supp{\frakI_{\phi}}=\AAA^2 $, and $\frakI_{\phi} $ is generated by some unions of the fibers of a $\phi$-equivariant map $\overline{x}:\AAA^2\rightarrow \AAA^1 $. Such $\phi $ has no torsion point. 
  
  \item Suppose $a$ is not root of unity. Then, $\frakI_{\phi} $ consists of three elements: the empty set $\emptyset $, a curve isomorphic to $\AAA^1 $, and the whole $\AAA^2 $. 
\end{enumerate}
\end{propp}
\begin{proof}
\begin{enumerate}[label=(\roman*)]
  \item Let $\phi^* $ act as a linear endomorphism on $V_{n+1} $ (see \Cref{Notation:dual-rep-of-jonquieres}). As a linear endomorphism $\phi^*-\id_{V_{n+1}}\in \End(V_{n+1}) $, it has the matrix representation as 
  \[
  \begin{pmatrix}
  0 & c & c^2 & \cdots & c^{n+1} & a_0 \\
  0 & 0 & 2c &  & (n+1)c^{n} & a_1 \\
   &  & 0 &  & \vdots &  \\
   &  &  & \ddots & (n+1)c & a_n \\
   & \text{\huge0} &  &  & 0 & 0 \\
    &  &  &  & 0 & 0
  \end{pmatrix} \, . 
  \]
  Hence, $(\phi^*-\id)(\kk \langle 1, y, \cdots, y^{n+1}\rangle)= \kk \langle 1, y, \cdots y^n \rangle $, and there exists $g(y)\in \kk[y]$ which satisfies $\deg{g(y)}=n+1 $ and $(\phi^*-\id)g(y)=P(y) $. Thus, introduce the change of coordinates $\hat{x}=x-g(y) $, $\hat{y}=y $, and let $\alpha=(\hat{x}, \hat{y})\in \Aut(\AAA^2) $ be the automorphism of change of coordinate. The new coordinates satisfy $\phi^* \hat{x}= \hat{x} $, $\phi^* \hat{y}= \hat{y}+c $, and under the change coordinates, $\hat{\phi}= \alpha \circ \phi \circ \alpha^{-1} $ is 
  \[
  \hat{\phi}(\hat{x}, \hat{y})= (\hat{x}, \hat{y}+c) \, . 
  \]
  
  Since a closed subvariety $C\subset \AAA^2 $ is invariant by $\phi $ if and only if $\alpha(C) $ is invariant by $\hat{\phi} $, we aim to describe the invariant subvariety lattice of $\hat{\phi} $ in the $(\hat{x}, \hat{y}) $-coordinate system. 

  Let $\alpha(p)=(\hat{x}_0, \hat{y}_0) $ be a point. Then, the $\hat{\phi}$-orbit of $\alpha(p) $ is contained in the $\hat{x} $-coordinate line $\hat{x}=\hat{x}_0 $, and coordinate line $\hat{x}=\hat{x}_0 $ contains infinitely many points of the $\hat{\phi} $-orbit of $\alpha(p) $. Thus, the Zariski closure of the $\hat{\phi} $-orbit of $\alpha(p) $ is the coordinate line $\hat{x}=\hat{x}_0 $. 

  It is clear that any union of $\hat{x} $-coordinate lines are invariant by $\hat{\phi} $. Conversely, suppose $\alpha(C)$ is closed subvariety which is invariant by $\hat{\phi}$. For each point $\alpha(p)\in \alpha(C)$, $\alpha(C)$ contains the whole $\hat{x} $-coordinate line through the point $\alpha(p)$. If $\alpha(C) $ contains a irreducible component whih is not a $\hat{x} $-coordinate line, then the image $\hat{x}(\alpha(C)) \subset \AAA^1$ is infinite. Hence, $\alpha(C) $ contains infinitely many $\hat{x}$-coordinate lines, which implies $\alpha(C)=\AAA^2 $. 
  
  Concluding, the $\hat{x}$-coordinate lines generate the invariant subvariety lattice of $\hat{\phi} $. It is also clear that $\hat{\phi} $ has no torsion points. 

  \item Inspect the matrix form of $\phi^*-a \id_{V_n} \in \End(V_n) $:
  \[
    \begin{pmatrix}
      1-a & c & \cdots &   & a_0 \\
       & 1-a &   & \textbf{\huge*} & \vdots \\
       &  & \ddots &  & a_{n-1} \\
      \text{\huge0} &   &  & 1-a & a_n \\
        &   &  &  & 0
      \end{pmatrix} \, . 
  \]
  Thus, arguing as in (i), there exists $\overline{x}=x-h(y)$ such that $\phi^* \overline{x}= a \overline{x} $ since $a\neq 1$. Let $\overline{y}=y $ and $\alpha=(\overline{x}, \overline{y}) \in \Aut(\AAA^2)$ be the automorphism of change of coordinate. Then, $\overline{\phi}= \alpha \circ \phi \circ \alpha^{-1} $ is 
  \begin{equation*}
    \overline{\phi}(\overline{x}, \overline{y})= (a\overline{x}, \overline{y}+c) \, .
  \end{equation*}
  
  Since a closed subvariety $C\subset \AAA^2 $ is invariant by $\phi $ if and only if $\alpha(C) $ is invariant by $\overline{\phi} $, we aim to describe the invariant subvariety lattice of $\overline{\phi} $ in the $(\overline{x}, \overline{y}) $-coordinate system. 

  Let $\alpha(p)=(\overline{x}_0, \overline{y}_0) $ be a point. The $\overline{\phi} $-orbit of $\alpha(p) $ is contained in the $\langle a \rangle$-orbit of the $\overline{x} $-coordinate line $\overline{x}=\overline{x}_0 $ (the root of unity $a$ acting as $a.\{ \overline{x}=\overline{x}_0\}= \{\overline{x}= a\overline{x}_0\} $), and each irreducible component of the $\langle a \rangle $-orbit of the $\overline{x} $-coordinate line $\overline{x}=\overline{x}_0 $ contains infinitely many points of the $\overline{\phi} $-orbit of $\alpha(p) $. Thus, Zariski closure of the $\overline{\phi} $-orbit of $\alpha(p) $ is the $\langle a \rangle$-orbit of the coordinate line $\overline{x}=\overline{x}_0 $. 

  Union of $\langle a \rangle $-orbit of $\overline{x} $-coordinate lines is obviously invariant by $\overline{\phi} $. Conversely, suppose $\alpha(C) $ is a closed subvariety which is invariant by $\overline{\phi} $. For each point $\alpha(p)\in \alpha(C)$, $\alpha(C) $ contains the $\langle a \rangle $-orbit of the $\overline{x} $-coordinate line through $\alpha(p) $. If $\alpha(C) $ contains a irreducible component which is not a $\overline{x} $-coordinate line, then the image $\overline{x}(\alpha(C)) \subset \AAA^1$ is infinite. Hence, $\alpha(C) $ contains infinitely many $\overline{x} $-coordinate lines, which implies $\alpha(C)=\AAA^2 $.

  Concluding, the invariant subvariety lattice of $\overline{\phi} $ is generated by the $ \langle a \rangle$-orbits of a $\overline{x}$-coordinate line. It is also obvious that $\overline{\phi} $ has no torsion points.

  \item We obtain the change of coordinates $\alpha $ and $\overline{\phi}(\overline{x}, \overline{y})=(a\overline{x}, \overline{y}+c) $ as in (ii). Since a closed subvariety $C\subset \AAA^2 $ is invariant by $\phi $ if and only if $\alpha(C) $ is invariant by $\overline{\phi} $, we aim to describe the invariant subvariety lattice of $\overline{\phi} $ in the $(\overline{x}, \overline{y}) $-coordinate system. 
  
  A closed subvariety is invariant by $\overline{\phi}$ if and only if it is invariant by the Zariski closure of $\langle \overline{\phi}\rangle \subset \GL_2(\kk) $. Since $a$ is not root of unity, the Zariski closure of the subgroup $\langle \overline{\phi} \rangle=\{\overline{\phi}^i(\overline{x}, \overline{y})= (a^n \overline{x}, \overline{y}+ nc)\, | \, n\in \ZZ \}\subset \GL_2(\kk) $ is 
  \[
    \overline{\langle \overline{\phi} \rangle }= \{\overline{\phi}_{t_1, t_2}(\overline{x}, \overline{y})= (t_1 \overline{x}, \overline{y}+ t_2)\, | \, t_1 \in \kk^{\times},\, t_2 \in \kk \} \subset \GL_2(\kk) 
  \]
  by the \Cref{lemma:transcendence} (ii). 
  
  Since $\frakI_{\overline{\phi}} = \frakI_{\overline{\langle \overline{\phi} \rangle}} \subset \frakI_{\overline{\phi}_{1,t}} $ where $\overline{\phi}_{1,t}= (\overline{x}, \overline{y}+t) $, any $\overline{\phi }$-invariant closed subvariety is invariant by $\overline{\phi}_{1,t} $. Arguing as in (i), any $\overline{\phi}_{1,t} $-invariant closed subvariety is a union of $\overline{x} $-coordinate lines. 

  It is obvious that the line $\overline{x}=0 $ is $\overline{\phi} $-invariant. If a closed subvariety $\alpha(C) $ contains a point outside the line $\overline{x}=0 $, then $\alpha(C) $ should contain infinitely many $\overline{x} $-coordinate lines which are the orbits of the point by $\overline{\phi}_{t,0}=(t \overline{x}, \overline{y}) $ and $\overline{\phi}_{1, t}=(\overline{x}, \overline{y}+t) $. Thus, the only nontrivial $\overline{\phi} $-invariant closed subvariety is the coordinate line $\overline{x}=0 $. 
\end{enumerate}
\end{proof}

Following proposition computes the invariant subvariety lattice of $\phi $ when one of $a$ or $b$ is a root of unity. The case when none of $a$ or $b$ is root of unity is dealt in \Cref{prop:computations-no-unity}. 
\begin{propp}
\label{prop:computations-unity}
Let $\phi=(ax+P(y), by+c)\in \JJ_n $ be such that $\deg{P(y)}=n$. If $b\neq 1 $, or $(b,c)=(1,0) $, then let the notations be as in \Cref{lemma:diagonalization}. Then, the following are true:
\begin{enumerate}[label=(\roman*)]
  \item Suppose that $a$ and $b$ are roots of unity, and that $h_1 = 0$. Then, $\phi $ is of finite order. 

  \item Suppose that $a$ and $b$ are roots of unity, and that $h_1\neq 0$. Then, $\Supp{\frakI_{\phi}}= \AAA^2 $, and $\frakI_{\phi} $ is generated by some unions of the $y$-coordinate lines $y=y_0 $, and the torsion points on the lines $h_1(y)=0 $. A point of $\AAA^2 $ is torsion point of $\phi $ if and only if it lies on the lines $h_1(y)=0 $.
  
  \item Suppose $a$ is not root of unity, and $b$ is a root of unity. Then, $\Supp{\frakI_{\phi}}= \AAA^2 $, and $\frakI_{\phi} $ is generated by the curve $\tilde{x}=0 $, some unions of $y$-coordinate lines $y=y_0 $, and the torsion points on the curve $\tilde{x}=0 $. A point of $\AAA^2 $ is torsion point of $\phi $ if and only if it lies on the curve $\tilde{x}=0 $.
  
  \item Suppose $a$ is root of unity and $b$ is not root of unity. Then, $\tilde{x} $ is $\phi$-equivariant, $\Supp{\frakI_{\phi}}=\AAA^2 $, and $\frakI_{\phi} $ is generated by some unions of the fibers of $\tilde{x} $, the line $\tilde{y}=0 $, and torsion points on the line $\tilde{y}=0 $. A point of $\AAA^2 $ is torsion point of $\phi $ if and only if it lies on the line $\tilde{y}=0 $. 
\end{enumerate}
\end{propp}
\begin{proof}
After applying the change of coordinates of \Cref{lemma:diagonalization}, a subvariety $C\subset \AAA^2 $ is invariant by $\phi $ if and only if $\alpha(C) $ is invariant by $\tilde{\phi} $. Thus, we aim to compute the invariant subvariety lattice for $\tilde{\phi} $ in the $(\tilde{x}, \tilde{y}) $-coordinate system. 
\begin{enumerate}[label=(\roman*)]
  \item The change of coordinates of \Cref{lemma:diagonalization} becomes $\tilde{\phi}(\tilde{x}, \tilde{y})= (a \tilde{x}, b \tilde{y}) $. Since $a$ and $b$ are roots of unity, $\tilde{\phi} $ is of finte order. 

  \item Let $\alpha(p)=(\tilde{x}_0, \tilde{y}_0) $ be a point. The description of the $\tilde{\phi} $-orbit of $\alpha(p) $ is provided by the \Cref{lemma:diagonalization} that $\tilde{\phi}^n (\alpha(p))= (a^n \tilde{x}_0 + n a^{n-1} \tilde{h}_1(\tilde{y_0}), b^n \tilde{y}_0 ) $. If $\alpha(p) $ is such that $\tilde{h}_1(\tilde{y}_0)=0 $, then $\tilde{\phi} $ acts on $\alpha(p) $ by $\tilde{\phi}^n(\alpha(p))= (a^n \tilde{x}_0, b^n \tilde{y}_0) $. Hence, all the points on the $\tilde{y} $-coordinate lines $\tilde{h}_1(\tilde{y}_0) =0$ are torsion points. 
  
  Suppose $\alpha(p) $ is such that $\tilde{h}_1(\tilde{y}_0)\neq 0 $. By the \Cref{lemma:diagonalization}, $\alpha(p)$ is not torsion point of $\tilde{\phi} $, $\tilde{\phi}$-orbit of $\alpha(p) $ is contained in the $\langle b \rangle $-orbit of $\tilde{y} $-coordinate lines $\tilde{y}=\tilde{y}_0 $, and each irreducible component of the $\langle b \rangle $-orbit of the line $\tilde{y}=\tilde{y}_0 $ contains infinitely many points among the $\tilde{\phi} $-orbit of $\alpha(p)$. Thus, the Zariski closure of the $\tilde{\phi} $-orbit of $\alpha(p) $ is the $\langle b \rangle $-orbit of the $\tilde{y} $-coordinate line $\tilde{y}=\tilde{y}_0 $. 
  
  The $\langle b \rangle $-orbit of the $\tilde{y} $-coordinate lines is obviously $\tilde{\phi} $-invariant. Conversely, suppose the closed subvariety $\alpha(C) $ is invariant by $\tilde{\phi} $. For each point $\alpha(p)\in \alpha(C) $ which doesn't lie on the curve $\tilde{h}_1(\tilde{y})=0 $, $\alpha(C) $ contains the whole $\tilde{y} $-coordinate line through $\alpha(p) $ and the $\langle b \rangle $-orbit of the line. If $\alpha(C) $ contains a dimension 1 irreducible component which is not a $\tilde{y} $-coordinate line, then the image $\tilde{y}(\alpha(C))\subset \AAA^1 $ is infinite, and $\alpha(C) $ contains infinitely many $\tilde{y} $-coordinate lines. Hence, $\alpha(C)=\AAA^2 $. 
  
  Concluding, any $\tilde{\phi} $-invariant curve is union of the $\langle b \rangle $-orbit of $\tilde{y} $-coordinate lines.

  \item Since $b$ is root of unity while $a$ is not, $\tilde{h}_1(\tilde{y})=\sum_{\substack{0 \leq i \leq n\\ a= b^i}} \tilde{a}_i \tilde{y}^i=0$ is an empty sum. Thus, $\phi $ is diagonalized as $\tilde{\phi}(\tilde{x}, \tilde{y})= (a \tilde{x}, b \tilde{y}) $. The Zariski closure of the subgroup $\langle \tilde{\phi} \rangle$ is semisimple, and we can analyze the affine GIT quotient to deduce the conclusion. Instead, we stick with the elementary argument that has use been used until now. 
  
  Suppose $\alpha(p)=(\tilde{x}_0, \tilde{y}_0) $ lies on the coordinate line $\tilde{x}=0 $, i.e. $\alpha(p)=(0, \tilde{y}_0) $. Since $b$ is root of unity, $\alpha(p) $ is torsion point of $\tilde{\phi}$. 
  
  Suppose $\alpha(p)$ does not lie on the coordinate line $\tilde{x}=0 $, i.e. $\tilde{x}_0\neq 0 $. Then, $\tilde{\phi}$-orbit of $\alpha(p) $ is contained in the $\langle b \rangle $-orbit of the $\tilde{y} $-coordinate line $\tilde{y}=\tilde{y}_0 $, and each irreducible component of the $\langle b \rangle $-orbit of the line $\tilde{y}=\tilde{y}_0 $ contains infinitely many points of the $\tilde{\phi} $-orbit of $\alpha(p) $. Thus, the Zariski closure of the $\tilde{\phi} $-orbit of $\alpha(p) $ is the $\langle b \rangle $-orbit of the $\tilde{y}$-coordinate line $\tilde{y}=\tilde{y}_0 $. In particular, $\alpha(p) $ is not torsion point of $\tilde{\phi} $. 

  It is obvious that the $\langle b \rangle $-orbit of a $\tilde{y} $-coordinate line is invariant by $\tilde{\phi} $. Conversely, suppose a closed subvariety $\alpha(C) $ contains an irreducible curve which is neither the line $\tilde{x}=0 $ nor a $\tilde{y} $-coordinate line. Then, the image $\tilde{y}(\alpha(C))\subset \AAA^1$ is infinite, and $\alpha(C) $ intersects with infinitely many $\tilde{y} $-coordinate lines. Since $\alpha(C) $ doesn't contain the line $\tilde{x}=0 $, the intersection of $\alpha(C) $ and a $\tilde{y} $-coordinate line is not on $\tilde{x}=0$ for almost all $\tilde{y}$-coordinate lines. As $\alpha(C) $ contains the orbit Zariski closure of a point in it, $\alpha(C) $ contains infinitely many $\tilde{y} $-coordinate lines, which implies that $\alpha(C)=\AAA^2 $. 

  Conluding, the line $\tilde{x}=0 $ and the $\tilde{y} $-coordinate lines are all the invariant curves for $\tilde{\phi} $ as in the statement.

  \item This is the same as in (iii) with the role of $a$ and $b$ reversed. Since $a$ is root of unity while $b$ is not, we have the empty sum $\tilde{h}_1(\tilde{y})=0 $, and $\tilde{\phi} $ is a diagonal map $\tilde{\phi}=(a\tilde{x}, b\tilde{y}) $ as in (iii). Simply, exchange $\tilde{x} $ and $\tilde{y} $ from the proof of (iii), and we have the desired conclusion. 
\end{enumerate}
\end{proof}

\begin{propp}
\label{prop:computations-no-unity}
Let $\phi =(ax+ P(y), by+c ) \in \JJ_n  $ be such that $\deg{P(y)}=n$. Suppose neither $a$ nor $b$ are roots of unity. Let the notations be as in \Cref{lemma:diagonalization}. Then, the following are true:
\begin{enumerate}[label=(\roman*)]
\item Suppose that $h_1=0 $, and that there exist nonzero integers $r_1$ and $r_2 $ such that $a^{r_1}=b^{r_2} $. Let $r_1, r_2 $ be the least such pair, and $r_1=ds_1 $, $r_2=d s_2 $ where $d=\mathrm{gcm}(r_1, r_2)$. Then, $\Supp{\frakI_{\phi}} = \AAA^2$, and $\frakI_{\phi} $ is generated by the some unions of the fibers of the $\phi $-equivariant map $\AAA^2 \dashrightarrow \PP^1: (x,y) \mapsto [\tilde{x}^{s_1}: \tilde{y}^{s_2}] $, and the unique fixed point $\tilde{x}=\tilde{y}=0 $.

\item Suppose that $h_1=0 $, and that there exist no nonzero integers $r_1 $ and $r_2 $ such that $a^{r_1}=b^{r_2} $. Then, $\frakI_{\phi} $ is finite, and is generated by three elements: the curve $\tilde{x}=0 $, the line $\tilde{y}=0 $, and the unique fixed point $\tilde{x}=\tilde{y}=0 $. 

\item Suppose $h_1\neq 0$. Then, $\frakI_{\phi}$ is finite, and is generated by two elements: the line $\tilde{y}=0 $, and the unique fixed point $\tilde{x}=\tilde{y}=0 $. 
\end{enumerate}
\end{propp}
\begin{proof}
As was done in \Cref{prop:computations-b-is-one} and \Cref{prop:computations-unity}, we apply the change of coordinate of \Cref{lemma:diagonalization}, and aim to compute the invariant subvarieties of $\tilde{\phi} $ in the $(\tilde{x}, \tilde{y}) $-coordinate system. 

\begin{enumerate}[label=(\roman*)]
\item By the assumption that $h_1=0 $, we immediately diagonalize as $\tilde{\phi}(\tilde{x}, \tilde{y})=(a\tilde{x}, b\tilde{y}) $. The Zariski closure of the subgroup $\langle \tilde{\phi} \rangle $ is semisimple, and we can analyze the projective GIT quotient to derive the conclusion. Instead, we stick with the more elementary argument used until now.

Define $\pi:\AAA^2 \dashrightarrow \PP^1 $ as $\pi(\tilde{x}, \tilde{y})=[\tilde{x}^{s_1}: \tilde{y}^{s_2}]$. Identifying $\pi$ as the rational function $\frac{\tilde{x^{s_1}}}{\tilde{y^{s_2}}} $ on $\AAA^2 $, we have 
\[
\tilde{\phi}^* \pi = \tilde{\phi}^* \frac{\tilde{x}^{s_1}}{\tilde{y}^{s_2}} = \frac{a^{s_1}}{b^{s_2}} \pi = \zeta_d \pi
\]
where $\zeta_d=\frac{a^{s_1}}{b^{s_2}} $ is the $d$-th root of unity. Hence, by the \Cref{lemma:equivar-projective}, $\pi $ is $\tilde{\phi} $-equivariant, and the action of $\tilde{\phi} $ on $\AAA^2 $ descends down to the finite order projective automorphism $([x:y]\mapsto [\zeta_d x:y]) \in \Aut(\PP^1) $. 

The point $(0,0) $ is obviously a fixed point of $\tilde{\phi} $. Let $\alpha(p)=(\tilde{x}_0, \tilde{y}_0) \neq (0,0)$. The $\tilde{\phi} $-orbit of $\alpha(p) $ is contained in the curves
\[
\begin{gathered}
\tilde{y}_0^{s_2} \tilde{x}^{s_1} - \tilde{x}_0^{s_1}\tilde{y}^{s_2} =0 \, , \\
\tilde{y}_0^{s_2} \tilde{x}^{s_1} - \tilde{x}_0^{s_1} \zeta_d \tilde{y}^{s_2} =0 \, , \\
\vdots \\
\tilde{y}_0^{s_2} \tilde{x}^{s_1} - \tilde{x}_0^{s_1} \zeta_d^{d-1} \tilde{y}^{s_2} =0 \, .
\end{gathered}
\]
The curves are the fibers by $\pi $ of the points in the $\tilde{\phi} $-orbit of $\pi(\alpha(p))\in \PP^1 $, and each of the curves contain infinitely many points among the $\tilde{\phi} $-orbit of $\alpha(p) $. Thus, Zariski closure of the $\tilde{\phi} $-orbit of $\alpha(p) $ is the union of above curves. 

Note that all the fibers of $\pi $ are irreducible since $s_1, s_2 $ are coprime, though a scheme-theoretic fiber may not be reduced. 

The fibers by $\pi $ of the points in the $\tilde{\phi} $-orbit of a point in $\PP^1 $ is obviously $\tilde{\phi} $-invariant. Conversely, let $\alpha(C) \subset \AAA^2$ be a closed subvariety which is $\tilde{\phi} $-invariant. If $(0,0)\neq \alpha(p)\in \alpha(C) $, then $\alpha(C) $ should contain the Zariski closure of the $\tilde{\phi} $-orbit of $\alpha(p) $. If $\alpha(C) $ contains a irreducible component which is not a fiber of $\pi $, then the image $\pi(\alpha(C))\subset \PP^1 $ is infinite, and $\pi(\alpha(C)) $ contains infinitely many fibers of $\pi $. Thus, $\alpha(C)=\AAA^2 $. 

Concluding, every $\tilde{\phi} $-invariant curve is a union of some fibers of $\pi$.  

\item By the assumption that $h_1=0 $, we immediately diagonalize as $\tilde{\phi}(\tilde{x}, \tilde{y})=(a \tilde{x}, b \tilde{y}) $. Since there exist no nonzero integers $r_1, r_2$ such that $a^{r_1}=b^{r_2} $, the Zariski closure of the subgroup $\langle \tilde{\phi} \rangle \subset \GL_2(\kk)$ is the torus $\overline{\langle \tilde{\phi} \rangle}= (\kk^\times)^2 $ (\Cref{lemma:transcendence} (i)). 

A closed subvariety is invariant by $\tilde{\phi}$ if and only if it is invariant by the Zariski closure $(\kk^{\times})^2 = \overline{\langle \tilde{\phi}\rangle} \subset \GL_2(\kk) $. The point $(0,0) $ is the unique closed orbit of $(\kk^{\times})^2 $. Then, the orbits $(\kk^{\times})^2.(1,0) $ and $(\kk^{\times})^2.(0,1) $ are the locally closed orbits whose  respective closures $\tilde{x}=0 $ and $\tilde{y} $ are 1-dimensional. Then, for any point $\alpha(p)=(\tilde{x}_0, \tilde{y}_0) $ such that $\tilde{x}_0\neq 0, \tilde{y}_0 \neq 0 $, the orbit is $(\kk^\times)^2.\alpha(p)= \AAA^2 - \{\tilde{x}\tilde{y}=0 \} $. Thus, we have conclusion of the statement.

\item The iterates of $\tilde{\phi} $ is $\tilde{\phi}^i(\tilde{x}, \tilde{y})= (a^i \tilde{x} +i a^{i-1} \tilde{h_1}(\tilde{y}), b^i \tilde{y})$. Since $h_1\neq 0$, there exists an integer $1\leq r \leq n$ such that $b^r=a$, and $\tilde{h}_1(\tilde{y})= \tilde{a}_r \tilde{y}^r $ in the notations of \Cref{lemma:diagonalization}. The Zariski closure of the subgroup $\langle \tilde{\phi} \rangle \subset \JJ_n  $ is 
\[
  \overline{\langle \tilde{\phi} \rangle}= \{\tilde{\phi}_{t_1,t_2}(\tilde{x},\tilde{y})=(t_1^r \tilde{x} + a^{-1} \tilde{a}_r t_2 \tilde{y}^r, t_1 \tilde{y})\, | \, t_1 \in \kk^{\times}, t_2 \in \kk \} \, . 
\]
by the \Cref{lemma:transcendence} (ii). Obviously, the $\tilde{y} $-coordinate line $\tilde{y}=0 $ is invariant by the closure $\overline{\langle \tilde{\phi} \rangle} $. 

Let $\alpha(p)=(\tilde{x}_0, \tilde{y}_0) $ be a such that $\tilde{y}_0\neq 0 $. Applying $\tilde{\phi}_{1,t}$ on $\alpha(p) $ for varying $t\in \kk $, the Zariski closure of the $\tilde{\phi} $-orbit of $\alpha(p) $ contains the whole $\tilde{y} $-coordinate line through $\alpha(p)$. Then, applying $\tilde{\phi}_{t,0} $ on the $\tilde{y} $-coordinate line through $\alpha(p) $ for varying $t\in \kk^{\times} $, the Zariski closure of the $\tilde{\phi} $-orbit of $\alpha(p) $ should be whole $\AAA^2 $. 

\par Thus, if a closed $\tilde{\phi}$-invariant closed subvariety is not the whole $\AAA^2 $, then it is contained in the line $\tilde{y}=0 $.

\end{enumerate}
\end{proof}

\begin{corop}
\label{coro:torsion-elt-distinct-exponents}
Suppose $\phi=(ax + P(y), by+c)\in \JJ_n $ where $a$ and $b$ are both roots of unity. Moreover, assume that $b\neq 1 $, or that $(b,c)=(1,0) $. If $a \neq 1, b, \cdots, b^n $, then $\phi $ is of finite order.
\end{corop}
\begin{proof}
This is the case (i) of \Cref{prop:computations-unity}. By the assumption that $a\neq 1, b, \cdots, b^n $, we have $h_1=0 $ in the notation of \Cref{lemma:diagonalization}. 
\end{proof}

\begin{corop}
\label{coro:torsion-elt}
Every torsion element in $\Aut(\AAA^2) $ is conjugate to a diagonal map. 
\end{corop}
\begin{proof}
In the amalgamated product, the torsion element is necessarily bounded. If not, the length increases by each power and it can't be a torsion element. Thus, the torsion elements in the amalgamated product are precisely those which are conjugate to a torsion element in the factors. The case when the automorphism is of finite order is precisely the case (i) of the \Cref{prop:computations-unity} (all other cases have an infinite orbit). 
\end{proof}

We classify the algebraic elements of $\Aut(\AAA^2) $ by their invariant subvariety lattice. A polynomial $f\in \kk[x,y] $ is called an \textit{$\AAA^1$-bundle projection} if all of its fibers are isomorphic to $\AAA^1 $ when considered as a morphism $f:\AAA^2 \rightarrow \AAA^1 $. 

\begin{definition}
\label{def:invar-subvar-lattice-classification}
Let $\phi \in \Aut(\AAA^2) $ be an algebraic element, that is, an element which is conjugate into $\Aff $ or $\JJ $. Suppose $\phi $ is not of finite order. Then, we have trichotomy for $\phi $ according to its invariant subvariety lattice:
\begin{enumerate}[label=(\roman*)]
\item $\phi $ is of \textit{orbit closure fibration type} if $\Supp{\frakI_{\phi}}= \AAA^2 $, and there exists a $\phi $-equivariant $\AAA^1 $-bundle projection map $\pi \in \kk[x,y] $ such that $\frakI_{\phi} $ is generated by some unions of the fibers of $\pi $, and, if there exists, torsion points, and a curve which is isomorphic to $\AAA^1 $ and transversal to the fibers of $\pi $.

\item $\phi$ is of \textit{projective quotient type} if $\Supp{\frakI_{\phi}}=\AAA^2 $, and there exists a $\phi $-equivariant rational map $\pi\in \kk(x,y)$ such that all the fibers of $\pi $, except possibly finitely many fibers, are irreducible and reduced, and some unions of fibers of $\pi $ together with the fixed point generate $\frakI_{\phi} $. In such the case, $\pi $ has a unique base point at the intersection of all the $\phi $-invariant curves. 

\item $\phi $ is of \textit{nonfibration type} if $\Supp{\frakI_{\phi}}\neq \AAA^2 $. 
\end{enumerate}

Polynomials $f_1, f_2\in \kk[x,y] $ whose general fibers are irreducible and reduced (thought as morphisms $f_i:\AAA^2 \rightarrow \AAA^1 $) are said to \textit{induce equivalent fibrations} if $f_1-f_2 \in \kk $. It is equivalent to the condition that all the fibers of $f_1$ and $f_2$ coincide, i.e. $f_1^{-1}(f_1(p))=f_2^{-1}(f_2(p)) $ for all $p\in \AAA^2 $. The equivalence of the two conditions is proved by arguing as in \Cref{lemma:equivar-affine}. 

Likewise, suppose rational functions $g_1, g_2 \in \kk(x,y) $ are such that $\dom(g_1)=\dom(g_2) \neq \AAA^2$, and their general fibers are irreducible and reduced (thought as rational maps $g_i:\AAA^2\dashrightarrow \PP^1 $). Then, we say $g_1 $ and $g_2 $ \textit{induce equivalent fibrations} if $g_1$ and $g_2 $ differ by a fractional linear transform. It is equivalent to the condition that all the fibers of $g_1$ and $g_2 $ coincide, i.e. $g_1^{-1}(g_1(p))= g_2^{-1}(g_2(p)) $ for all $p\in \dom(g_i)$. The equivalence of the two conditions is proved by arguing as in \Cref{lemma:equivar-projective}.

We call the $\phi $-equivariant rational map $\pi $ as in (i) or (ii) the \textit{associated equivariant map of $\phi$}. 

Suppose automorphisms $\phi_1$ and $\phi_2 $ have the same type, and are not of nonfibration type. Let $\pi_1 $ and $\pi_2 $ be their associated equivariant maps. We say that $\phi_1 $ and $\phi_2 $ \textit{have equivalent fibrations} if their associated equivariant maps induce equivalent fibrations.
\end{definition}

\includegraphics[scale=0.75]{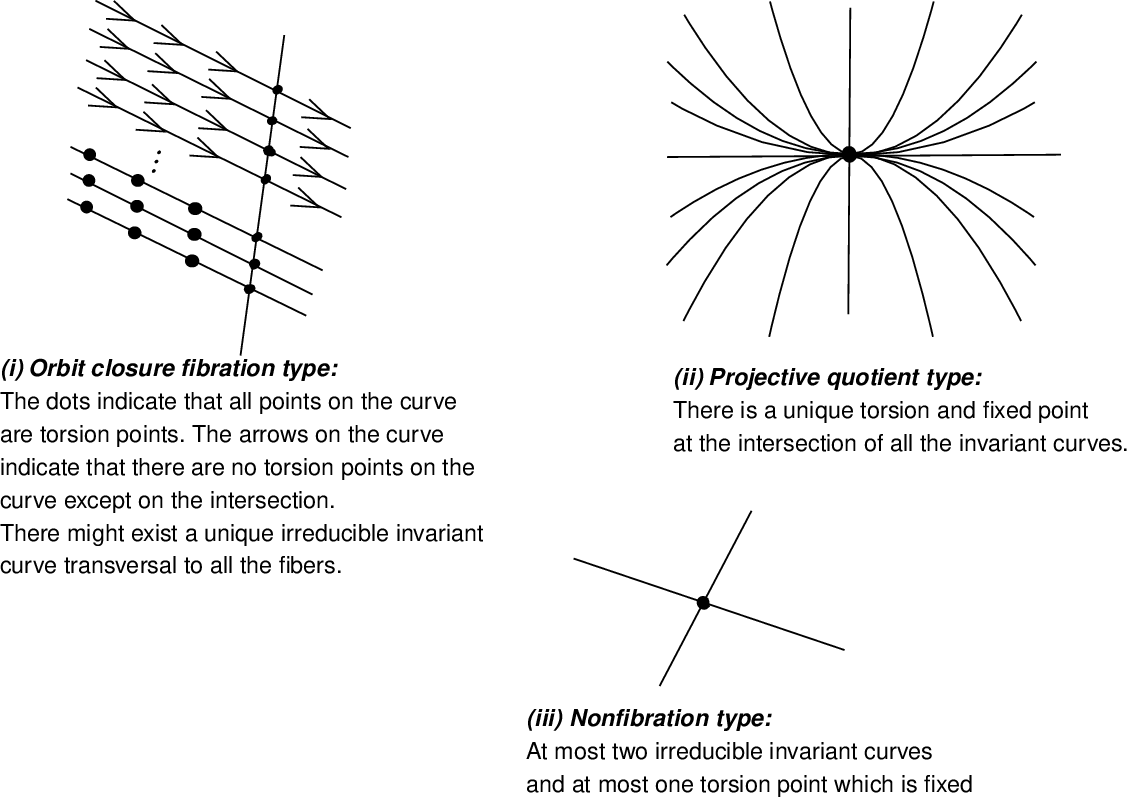}

\begin{rmk}
\label{rmk:fibration-types}
\begin{enumerate}[label=(\roman*)]
\item The orbit fibration types are (i), (ii) of \Cref{prop:computations-b-is-one}, and the instances of \Cref{prop:computations-unity}. The projective quotient type is (i) of the \Cref{prop:computations-no-unity}. The nonfibration types are (iii) of the \Cref{prop:computations-b-is-one}, and (ii), (iii) of the \Cref{prop:computations-no-unity}. The associated equivariant maps are given by taking appropriate coordinate map as in the computations propositions \Cref{prop:computations-b-is-one}, \Cref{prop:computations-unity} and \Cref{prop:computations-no-unity}.

\item Suppose $\phi_1 $ is of orbit closure fibration type while $\phi_2 $ is of nonfibration type. Then $\frakI_{\langle \phi_1, \phi_2 \rangle}= \frakI_{\phi_1}\cap \frakI_{\phi_2} $ is readily computed since $\frakI_{\phi_2} $ is finite. Simply check whether each element of $\frakI_{\phi_2} $ is $\phi_1 $-invariant. 

Now, suppose $\phi_2 $ is of projective quotient type. Then, $\frakI_{\langle \phi_1, \phi_2 \rangle}=\frakI_{\phi_1}\cap \frakI_{\phi_2} $ is again easily computed. All elements of the $\frakI_{\phi_2} $ contains the base point of the associated equivariant map. However, for a given point, there are most three elements through the point in $\frakI_{\phi_1} $. Examining those elements of $\frakI_{\phi_1} $ which pass through the unique base point of associate equivariant map of $\phi_2 $, we easily compute $\frakI_{\langle \phi_1, \phi_2 \rangle}=\frakI_{\phi_1}\cap \frakI_{\phi_2} $.

\item Suppose $\phi_1 $ and $\phi_2 $ have the same types and equivalent fibrations. Let $\pi_1 $ and $\pi_2 $ be the associated equivariant maps for $\phi_1 $ and $\phi_2 $ respectively. Then, $\pi_1 $ and $\pi_2 $ are $\langle \phi_1, \phi_2 \rangle $-equivariant by \Cref{lemma:equivar-affine} and \Cref{lemma:equivar-projective}.

\item Suppose $\phi$ is of finite order. Then, there exists a coordinate change as in \Cref{lemma:diagonalization} such that $a$ and $b$ are roots of unity. Suppose $a$ and $b$ are primitive $s_1 $-th and $s_2$-th roots of unity, respectively. Then, we have the GIT quotient map 
\[
  \AAA^2 \rightarrow \AAA^2 \sslash \langle \phi \rangle= \Spec{\kk[\tilde{x}^{s_1}, \tilde{y}^{s_2}]} \, .
\]
For a given point, there are infinitely many minimal $\phi $-invariant curves passing through the point. For this reason, the automorphisms of finite order will be dealt separately from the non-torsion elements. 
\end{enumerate}
\end{rmk}

\section{Computation of Orbit Zariski Closure}
\label{section:computation-orbit-closure}

In this section, we give an account of the development from Selberg's Lemma to Whang's universal orbit bound theorem. Then, I sum up everthing from previous sections to prove the decidability of algebraic orbit problem in $\AAA^2 $. The main ingredients for the proof of decidability are the Blanc-Stampfli Theorem (\Cref{prop:blanc-stampfli}), the conjugacy problem algorithm in $\AAA^2 $ (\Cref{theo:conj-criterion-aut}) and the uniform orbit bound theorem of Whang (\Cref{prop:whang}) \cite{Whang}.

\par Now, we gather everything from the previous sections together to compute the orbit closure of a finitely generated group action on $\AAA_{\overline{\QQ}}^2 $. We record the following easy lemma to use in \Cref{theo:orbit-zariski-closure}. 

\begin{lemmap}
\label{lemma:ozc_dimension_1}
Suppose $X\subset \AAA^2 $ is a variety which is of dimension 1 and $G\subset \Aut(X) $ a finitely generated group. For a given point $x\in X $, there exists an algorithm to compute the orbit Zariski closure of $x \in X$. 
\end{lemmap}
\begin{proof}
Using the algorithm of \Cref{prop:whang}, check whether $x$ is periodic with respect to $G$. If it is periodic, then we know that the $\overline{G.x} $ is 0-dimensional and the size of the orbit is smaller than the bound of \Cref{prop:whang}. Thus, we have computed orbit Zariski closure of $x$. 

\par Suppose $x$ is not periodic with respect to $G$. The orbit Zariski closure $\overline{G.x} \subset X$ should be of pure dimension 1 by the transitivity. There exists an irreducible component of $X$ which contains infinitely many points of $G.x $, say $X_1\subset X $. Then, $\overline{G.x}=\overline{G.X_1} $ which is of pure dimension 1.

\par Since there are finitely many subvarieties of pure dimension 1, by computing images of the irreducible components by the generators of $G$, it is possible to decide all $G$-invariant subvarieties of $X$. The smallest among those containing $x$ is the orbit Zariski closure in question. 
\end{proof}

\begin{theop}
\label{theo:orbit-zariski-closure}
Let $G$ be a finitely generated group $G$ acting algebraically on $\AAA_{\Qbar}^2 $. Then, given a point $p\in \AAA^2(\Qbar) $, there exists an algorithm to compute the orbit Zariski closure of $p$. 
\end{theop}
\begin{proof}
Identify $G $ as finitely generated subgroup of $\Aut(\AAA^2) $. We demonstrate the algorithm to compute the orbit Zariski closure $\overline{G.p} $, and argue the validity of the algorithm together. Through the algorithm, for $g \in \Aut(\AAA^2)$ which is in $\Aff \cup \JJ $, its type, the associated equivariant map (\Cref{def:invar-subvar-lattice-classification}), or the invariant subvariety lattice (\Cref{def:invar-subvar-lattice}) are calculated using the computations of \Cref{prop:computations-b-is-one}, \Cref{prop:computations-unity} and \Cref{prop:computations-no-unity}.  
\begin{enumerate}[label=\textbf{Step \arabic*.}]
  \item Use \Cref{prop:whang} to decide whether $\overline{G.p} $ is 0-dimensional or not. If it is not 0-dimensional, go to \textbf{Step 2}. If $\overline{G.p} $ is 0-dimensional, we can compute the Zariski closure of $G.p $ after finitely many computations and terminate. The cardinality of the orbit is less than the universal bound of \Cref{prop:whang}.
  
  \item From this step on, $\overline{G.p} $ is not 0-dimensional. Use \Cref{theo:conj-criterion-aut} to check whether $G $ is bounded or not. If $G$ is bounded, go to \textbf{Step 3}. If $G$ is unbounded, then it doesn't preserve any curves by \Cref{coro:blanc-stampfli}. Thus, $\overline{G.p} $ is not 1-dimensional. Conclude that $\overline{G.p}= \AAA^2 $ and terminate.  

  \item From this step on, $\overline{G.p} $ is not 0-dimensional and $G$ is bounded. Hence, $\overline{G.p}$ is of pure dimension 1 by transitivity, or $\overline{G.p}=\AAA^2 $.  
  
  \par Use \Cref{theo:conj-criterion-aut} to conjugate $G$ into either $\Aff $ or $\JJ $. Replace and denote this conjugation by the same $G=\langle g_1,\cdots, g_n \rangle $. Replace and denote by same $p$ the image of $p$ in the new conjugated coordinate system. Each $g_i $ is an algebraic element and we can decide whether it is torsion or not (see the proof of \Cref{coro:torsion-elt}). 
  
  \par Suppose all $g_1,\cdots, g_n $ are torsion elements. Since it was assumed that $G.p $ is infinite, $G$ is an infinite linear group. The Schur's Theorem on the Burnside problem for linear groups (\cite{Schur}) tells that an infinite linear group should contain an element of infinite order. Thus, by enumerating the words in $g_1,\cdots, g_n $, we produce an element of infinite order. After adding the element of infinite order in the generator set of $G$, we may assume that at least one of $g_i $'s is of infinite order. 

  \par Reindex the generators so that $g_1,\cdots, g_r $ are non-torsion, and $g_{r+1},\cdots, g_{n} $ are torsion elements.

  \item Check whether there exists a $g_i $ ($i\leq r $) which is of nonfibration type. If there is no such $g_i $, go to \textbf{Step 5}. If there is such $g_i $, then $\Supp{\frakI_G} \subset \Supp{\frakI_{g_i}}\neq \AAA^2 $. Compute the subvariety $\Supp{\frakI_{g_i}} $, set $X=\Supp{\frakI_{g_i}} $, and then go to \textbf{Step 8}.

  \item Check whether some $g_i$ and $g_j$ have distinct types or inequivalent fibrations for some $i,j \leq r $ (\Cref{def:invar-subvar-lattice-classification}). If all of them have the same types and equivalent fibrations, go to \textbf{Step 6}. Else, $\Supp{\frakI_{G}}\subset \Supp{(\frakI_{g_i} \cap \frakI_{g_j})} \neq \AAA^2 $, and we can compute $\Supp{(\frakI_{g_i} \cap \frakI_{g_j})} $ (\Cref{rmk:fibration-types} (ii)). Proceed to \textbf{Step 8} with $X=\Supp{\frakI_{g_i}\cap \frakI_{g_j}} $. 

  \item From this step on, assume that all $g_1, \cdots, g_r $ have the same types and fibrations. By the \Cref{rmk:fibration-types} (iii), if $\pi$ is an associated equivariant map for one of $g_1,\cdots, g_r $, then $\pi$ is $\langle g_1,\cdots, g_r\rangle $-equivariant. 
  
  If $g_1,\cdots, g_r $ are of projective quotient type, go to \textbf{Step 7}. In this step, we deal with the case when $g_1,\cdots, g_r $ are of orbit closure fibration type with equivalent fibrations. 
  
  For each $i=1,\cdots, r$, let $L_i$ denote (if it exists) the unique $g_i $-curve transversal to all the fibers of $\pi $ (see the diagram below \Cref{def:invar-subvar-lattice-classification}). If there is no such curve corresponding to $g_i $, then set $L_i=\emptyset $. If $L_1=\cdots = L_r $, then let $L=L_1$. Else, let $L=\emptyset $. 
  
  \begin{enumerate}[label=\textbf{Step 6-\arabic*.}]
    \item Check whether $\pi $ is equivariant by each $g_{r+1},\cdots, g_n $ (\Cref{lemma:equivar-affine}). If $\pi $ is equivariant by all of $g_{r+1},\cdots, g_n $, then go to \textbf{Step 6-2}.    
    
    Suppose otherwise that there exists $j>r $ such that $g_j^* \pi $ is not linear in $\pi $. We claim that $\overline{G.p}\subset L_1\cup g_j.L_1 $ or $\overline{G.p}= \AAA^2 $. No union of fibers of $\pi $ is invariant by $g_j $ since the image $\pi(g_j.\pi^{-1}(t))\subset \AAA^1 $ is infinite for any fiber $\pi^{-1}(t) $ (\Cref{lemma:equivar-affine}). Since it was assumed from \textbf{Step 3} that $\dim{\overline{G.p}}\neq 0 $, following the description of $\frakI_{g_1}$ in \Cref{def:invar-subvar-lattice-classification}, every invariant subvariety of $g_1 $ is some union of fibers of $\pi $ or $L_1 $, and $\overline{G.p}\subset L_1\cup \bigcup_{i=1}^k \pi^{-1}(t_i) $ or $\overline{G.p}=\AAA^2 $ where $\pi^{-1}(t_i) $'s are some finitely many fibers of $\pi $. In fact, here we can take $k=1 $, as $L_1\cup \pi^{-1}(t_1)\cup \pi^{-1}(t_2) $ cannot be $g_j $-invariant for $t_1\neq t_2 $. If it were, then $g_j.\pi^{-1}(t_1)=g_j.\pi^{-1}(t_2)=L_1 $, which follows from and contradicts \Cref{lemma:equivar-affine}. Thus, we have demonstrated $\overline{G.p}\subset L_1\cup g_j.L_1 $ or $\overline{G.p}=\AAA^2 $.
    
    Find the smallest pure dimension 1 subvariety of $L_1\cup g_j.L_1 $ which is $G$-invariant and contains $p$. If such the subvariety exists, then conclude that $\overline{G.p}$ is the subvariety and terminate. If no such the subvariety exists, then conclude $\overline{G.p}=\AAA^2 $ and terminte.

    \item Since $\pi$ is $\langle g_{r+1},\cdots, g_n\rangle $-equivariant, $\pi$ is $G$-equivariant. By \Cref{lemma:equivar-affine}, the action of $G$ on $\AAA^2 $ descends down to the algebraic action on $\AAA^1\cong \Spec{\kk[\pi]} $. Use \Cref{prop:whang} to decide whether $\pi(p)\in \AAA^1 $ is periodic or not. If $\pi(p)$ is periodic, then $G.\pi(p) $ is effectively computed and $\pi^{-1}(G.\pi(p)) $ is $G$-invariant. Use \Cref{lemma:ozc_dimension_1} on $\pi^{-1}(G.\pi(p)) $ to conclude and terminate. 
    
    If $\pi(p) $ is not periodic, then we claim that $\overline{G.p}=L $ or $\overline{G.p}=\AAA^2 $. If $G.p$ contains a point $q$ outside $L$, then $\overline{G.p} $ contains the whole fiber $\pi^{-1}(\pi(q)) $ following the description of $\frakI_{g_1} $ from \Cref{def:invar-subvar-lattice-classification} and the assumption that $\dim{\overline{G.p}}\neq 0 $ from \textbf{Step 3}. Then, $G.\pi^{-1}(\pi(q)) $ is union of infinitely many fibers of $\pi $. Since $G.\pi^{-1}(\pi(q)) \subset \overline{G.p} $, the closure is $\overline{G.p}=\AAA^2 $. 

    \par Following the argument, conclude that $\overline{G.p}=L $ or $\AAA^2 $ after examining whether $L$ is $G$-invariant and $p\in L $, then terminate.  
  \end{enumerate}

  \item In this step, we deal with the case when $g_1, \cdots, g_r$ are of projective quotient type with equivalent fibrations. The fiber $\pi^{-1}(t) $ denotes the strict transform.
  
  Check whether the unique base point of $\pi $ is fixed by each $g_{r+1},\cdots, g_{n} $. If the unique base point of $\pi $ is fixed by all $g_{r+1},\cdots, g_n $, go to \textbf{Step 7-1}. 
  
  If there exists some $g_j $ ($j>r $) such that $g_j $ doesn't fix the unique base point of $\pi $, denote by $q_1,\cdots, q_l $ the points in the $g_j $-orbit of the base point with $q_1$ the base point of $\pi$.

  Suppose a fiber $\pi^{-1}(t)$ contains none of $q_2,\cdots, q_l $. Then, $g_j.\pi^{-1}(t) $ doesn't contain $q_1 $ which is the base point of $\pi $. Hence, $g_j.\pi^{-1}(t) $ is not a fiber of $\pi $, and the image $\pi(g_j.\pi^{-1}(t))\subset \PP^1 $ is infinite. By the description of $\frakI_{g_1} $ from \Cref{def:invar-subvar-lattice-classification}, every invariant subvariety of $g_1 $ is some union of fibers of $\pi $, and thus, there exists no nontrivial $G$-invariant subvariety containing $\pi^{-1}(t) $. We have proved that 
  \[
  \Supp{\frakI_{G}} \subset \Supp{(\frakI_{\langle g_1,\cdots, g_r \rangle} \cap \frakI_{g_j})} \subset \bigcup_{i=1}^l \pi^{-1}(\pi(q_i)) \, . 
  \]
  Compute $X=\bigcup_{i=1}^l \pi^{-1}(\pi(q_i)) $ and go to \textbf{Step 8}. 

  \begin{enumerate}[label=\textbf{Step 7-\arabic*.}]
    \item Check whether $\pi $ is equivariant by each $g_{r+1},\cdots, g_n $ (\Cref{lemma:equivar-projective}). If $\pi $ is equivariant by all of $g_{r+1},\cdots, g_n $, then go to \textbf{Step 7-2}.
    
    Otherwise, let $J=\{j>r \, \vert \, \text{$\pi $ is not $g_j $-equivariant} \}$. For each $j\in J$, there exists at most a pair $t_{j1}, t_{j2}\in \PP^1 $ such that the fibers $\pi^{-1}(t_{j1}) $, $\pi^{-1}(t_{j2}) $ are irreducible and reduced, and $g_j.\pi^{-1}(t_{j1})= \pi^{-1}(t_{j2}) $ (\Cref{lemma:equivar-projective}). Denote by $Z\subset \AAA^2$ the underlying reduced closed subscheme of the non-reduced fibers of $\pi $. From the description of the $\frakI_{g_1} $ provided by \Cref{def:invar-subvar-lattice-classification}, only union of fibers of $\pi $ can be nontrivial $G $-invariant subvariety. Hence, 
    \[
    \Supp{\frakI_{G}}\subset \bigcup_{j\in J} (\pi^{-1}(t_{j1})\cup \pi^{-1}(t_{j2})) \cup Z \, . 
    \]
    Compute $X=\bigcup_{j\in J} (\pi^{-1}(t_{j1})\cup \pi^{-1}(t_{j2})) \cup Z  $, and go to \textbf{Step 8}.

    \item Here, we assume that $\pi $ is $G$-equivariant. Hence, by \Cref{lemma:equivar-projective}, the action of $G$ on $\AAA^2 $ descends down to action on $\PP^1 $. Check whether $\pi(p) $ is $G$-periodic or not by \Cref{prop:whang}. If it is not periodic, then $\overline{G.p} $ contains infinitely many fibers of $\pi$. Conclude $\overline{G.p}=\AAA^2 $ and terminate. 
    
    If $\pi(p) $ is periodic, then $\overline{G.p}\subset \pi^{-1}(G.\pi(p)) $, and $\pi^{-1}(G.\pi(p))$ is $G$-invariant. Compute $\overline{G.p} $ using \Cref{lemma:ozc_dimension_1}, then terminate. 
  \end{enumerate}

  \item We have computed a subvariety $X\subsetneq \AAA^2 $ such that $\Supp{\frakI_{G}}\subset X $ from some previous step. Find the smallest pure dimension 1 subvariety of $X $ which is preserved by all $g_1,\cdots, g_n $. Since there are finitely many 1-dimensional subvarieties of $X $, this can be computed in finitely many steps.
  
  If there is such the 1-dimensional subvariety of $X$ which contains $x$, then we conclude that $\overline{G.p} $ is the subvariety and terminate. If there is no such the 1-dimensional subvariety of $X$, then we conclude $\overline{G.p}=\AAA^2 $ and terminate.
\end{enumerate}
\end{proof}

\printbibliography

\end{document}